\documentclass[reqno,11pt]{amsart}
\usepackage{amssymb}
\usepackage{verbatim}
\newif\ifpdf
\ifpdf
  \usepackage[pdftex]{graphicx}
  \usepackage[pdftex]{hyperref}
\else
  \usepackage{graphicx}
\fi
\numberwithin{equation}{section}       % Number formulas within sections
 \theoremstyle{plain}    
 \newtheorem{thm}{Theorem}[section]
 \numberwithin{equation}{section} %% Comment out for sequentially-numbered
 \numberwithin{figure}{section} %% Comment out for sequentially-numbered
 \theoremstyle{plain}
 \theoremstyle{plain}    
 \newtheorem{cor}[thm]{Corollary} %%Delete [thm] to re-start numbering
 \theoremstyle{plain}    
 \newtheorem{pro}[thm]{Proposition} %%Delete [thm] to re-start numbering
 \theoremstyle{plain}    
 \newtheorem{lem}[thm]{Lemma} %%Delete [thm] to re-start numbering
 \theoremstyle{remark}
 \newtheorem{rem}[thm]{Remark}
 \theoremstyle{definition}
\newtheorem{ex}[thm]{Example}
\newtheorem*{thmA}{Theorem A} 
\newtheorem*{thmB}{Theorem B} 
\newtheorem*{thmC}{Theorem C} 
\newtheorem*{thmD}{Theorem D}

\theoremstyle{definition}
\newtheorem{defi}[thm]{Definition}

\newtheorem*{ackn}{Acknowledgement}

\newtheorem{conjecture}[thm]{Conjecture}

\newcommand{\C}{{\mathbb{C}}}

\newcommand{\PP}{{\mathbb{P}}}
\newcommand{\Q}{{\mathbb{Q}}}
\newcommand{\R}{{\mathbb{R}}}

\newcommand{\Z}{{\mathbb{Z}}}

\newcommand{\cC}{{\mathcal{C}}}

\newcommand{\cE}{{\mathcal{E}}}

\newcommand{\cI}{{\mathcal{I}}}

\newcommand{\cO}{{\mathcal{O}}}

\newcommand{\cM}{{\mathcal{M}}}

\newcommand{\cK}{{\mathcal{K}}}

\renewcommand{\a}{\alpha}

\renewcommand{\d}{\delta}
\newcommand{\e}{\varepsilon}
\newcommand{\f}{\varphi}
\newcommand{\fm}{\varphi_{\min}}
\newcommand{\p}{\psi}

\newcommand{\MA}{\mathrm{MA}\,}
\newcommand{\Amp}{\mathrm{Amp}\,}
\newcommand{\xcan}{X_{\mathrm{can}}}

\newcommand{\vol}{\operatorname{vol}}

\newcommand{\ca}{\operatorname{Cap}}

\begin{document}

\setcounter{tocdepth}{1}

\title{Monge-Amp{\`e}re equations in big cohomology classes}

\date{1 December 2009}

\author{S.~Boucksom, P.~Eyssidieux, V.~Guedj, A.Zeriahi}

\address{IMJ, CNRS-Universit{\'e} Paris 7\\
 75251 Paris Cedex 05\\
 France}

\email{boucksom@math.jussieu.fr}

\address{Institut Fourier, Universit{\'e} Grenoble 1\\
38402 Saint-Martin d'H{\`e}res\\
France}

\email{Philippe.Eyssidieux@ujf-grenoble.fr}

\address{LATP, Universit{\'e} Aix-Marseille I\\
13453 Marseille Cedex 13\\
 France}

\email{guedj@cmi.univ-mrs.fr}

\address{I.M.T., Universit{\'e} Paul Sabatier\\
31062 Toulouse cedex 09\\
France}

\email{zeriahi@math.ups-tlse.fr}

\begin{abstract}
We define non-pluripolar products of an arbitrary number of closed positive $(1,1)$-currents on a compact K\"ahler manifold $X$. Given a big $(1,1)$-cohomology class $\a$ on $X$ (i.e.~a class that can be represented by a strictly positive current) and a positive measure $\mu$ on $X$ of total mass equal to the volume of $\a$ and putting no mass on pluripolar sets, we show that $\mu$ can be written in a unique way as the top degree self-intersection in the non-pluripolar sense of a closed positive current in $\a$. We then extend Kolodziedj's approach to sup-norm estimates to show that the solution has minimal singularities in the sense of Demailly if $\mu$ has $L^{1+\e}$-density with respect to Lebesgue measure. If $\mu$ is smooth and positive everywhere, we prove that $T$ is smooth on the ample locus of $\a$ provided $\a$ is nef. Using a fixed point theorem we finally explain how to construct singular K\"ahler-Einstein volume forms with minimal singularities on varieties of general type.

\end{abstract} 

\maketitle

\tableofcontents

\newpage

\section*{Introduction}

The primary goal of the present paper is to extend the results obtained by the last two authors on complex Monge-Amp{\`e}re equations on compact K{\"a}hler manifolds in~\cite{GZ2} to the case of \emph{arbitrary} cohomology classes. 

More specifically let $X$ be a compact $n$-dimensional K{\"a}hler manifold and let $\omega$ be a K{\"a}hler form on $X$. An $\omega$-plurisubharmonic (\emph{psh} for short) function is an upper semi-continuous function $\f$ such that $\omega+dd^c\f$ is positive in the sense of currents, and~\cite{GZ2} focused on the Monge-Amp{\`e}re type equation
\begin{equation}\label{equ:MA}(\omega+dd^c\f)^n=\mu
\end{equation}
where $\mu$ is a positive measure on $X$ of total mass $\mu(X)=\int_X\omega^n$. As is well-known, it is not always possible to make sense of the left-hand side of (\ref{equ:MA}), but it was observed in~\cite{GZ2} that a construction going back to Bedford-Taylor~\cite{BT2} enables in this global setting to define the non-pluripolar part of the would-be positive measure $(\omega+dd^c\f)^n$ for an arbitrary $\omega$-psh function $\f$. 

\smallskip

In the present paper we first give a more systematic treatment of these issues. We explain how to define in a simple and canonical way the \emph{non-pluripolar product} $$\langle T_1\wedge...\wedge T_p\rangle$$ 
of arbitrary closed positive $(1,1)$-currents $T_1,...,T_p$ with $1\le p\le n$. The resulting positive $(p,p)$-current $\langle T_1\wedge...\wedge T_p\rangle$ puts no mass on pluripolar subsets. It is also shown to be \emph{closed} (Theorem~\ref{thm:closed}), generalizing a classical result of Skoda-El Mir.
We relate its cohomology class to the \emph{positive intersection class}
$$\langle\a_1\cdot \cdots \cdot\a_p\rangle\in H^{p,p}(X,\R)$$
of the cohomology classes $\a_j:=\{T_j\}\in H^{1,1}(X,\R)$ in the sense of~\cite{BDPP,BFJ}. In particular we show that $\int_X\langle T^n\rangle\le\vol(\a)$
where the right-hand side denotes the \emph{volume} of the class $\a:=\{T\}$~\cite{Bou1}, which implies in particular that $\langle T^n\rangle$ is non-trivial only if the cohomology class $\a$ is \emph{big}. 

\smallskip

An important aspect of the present approach is that the non-pluripolar Monge-Amp\`ere measure
$\langle T^n \rangle$ is well defined for {\it any} closed positive $(1,1)$-current $T$. In the 
second section we study the continuity properties of the mapping
$T \mapsto \langle T^n \rangle$: it is continuous along decreasing sequences (of potentials)
if and only if $T$ {\it has full Monge-Amp\`ere mass} (Theorem\ref{thm:cont}), i.e. when
$$
\int_X  \langle T^n \rangle=\rm{vol}(\a).
$$

We prove this fact defining and studying weighted energy functionals in this general context, extending the case of a K\"ahler class~\cite{GZ2}. The two main new features are
a generalized comparison principle (Corollary \ref{cor:comp}) and an asymptotic criterion to check whether a current $T$ has full Monge-Amp\`ere mass (Proposition \ref{pro:presquefini}).

\smallskip

In the third part of the paper we obtain our first main result (Theorem \ref{thm:MA}):

\begin{thmA}
{\it  Let $\a\in H^{1,1}(X,\R)$ be a big class on a compact K{\"a}hler manifold $X$. If $\mu$ is a positive measure on $X$ that puts no mass on pluripolar subsets and satisfies the compatibility condition $\mu(X)=\vol(\a)$, then there exists a unique closed positive $(1,1)$-current $T\in\a$ such that 
$$\langle T^n\rangle=\mu.$$}
\end{thmA}

The existence part extends the main result of~\cite{GZ2}, which corresponds exactly to the case where $\a$ is a K{\"a}hler class. In fact the proof of Theorem A consists in reducing to the K{\"a}hler case \emph{via} approximate Zariski decompositions. Uniqueness is obtained by adapting the proof of S.~Dinew~\cite{Din2} (which also deals with the K{\"a}hler class case). 

\medskip

When the measure $\mu$ satisfies some additional regularity condition, we show how to adapt Ko\l{}odziej's pluripotential theoretic approach to the sup-norm \emph{a priori} estimates~\cite{Ko2} to get \emph{global} information on the singularities of $T$.

\begin{thmB} 
{\it Assume that the measure $\mu$ in Theorem A furthermore has $L^{1+\e}$ density with respect to Lebesgue measure for some $\e>0$. Then the solution $T\in\a$ to $\langle T^n\rangle=\mu$ has \emph{minimal singularities}.}
\end{thmB}

Currents with minimal singularities were introduced by Demailly. When $\a$ is a K{\"a}hler class, the positive currents $T\in\a$ with minimal singularities are exactly those with \emph{locally bounded} potentials. When $\a$ is merely big all positive currents $T\in\a$ will have poles in general, and the minimal singularity condition on $T$ essentially says that $T$ has the least possible poles among all positive currents in $\a$. Currents with minimal singularities have in particular locally bounded potentials on the \emph{ample locus} $\Amp(\a)$ of $\a$, which is roughly speaking the largest Zariski open subset where $\a$ locally looks like a K{\"a}hler class. 

\smallskip

Regarding local regularity properties, we obtain the following result.

\begin{thmC} 
{\it In the setting of Theorem A, assume that $\mu$ is a smooth strictly positive volume form. Assume also that $\a$ is \emph{nef}. Then the solution $T\in\a$ to the equation $\langle T^n\rangle=\mu$ is $C^\infty$ on $\Amp(\a)$. }
\end{thmC}

The expectation is of course that Theorem C holds whether or not $\a$ is nef, but we are unfortunately unable to prove this for the moment.  It is perhaps worth emphasizing that currents with minimal
singularities can have a non empty polar set even when $\a$ is nef and big
(see Example \ref{ex:bignef}).

\medskip

In the last part of the paper we consider Monge-Amp{\`e}re equations of the form
\begin{equation}\label{equ:MAKE}\langle(\theta+dd^c\f)^n\rangle=e^\f dV
\end{equation}
where $\theta$ is a smooth representative of a big cohomology class $\a$, $\f$ is a $\theta$-psh function and $dV$ is a smooth positive volume form. We show that (\ref{equ:MAKE}) admits a unique solution $\f$ such that $\int_X e^\f dV=\vol(\a)$. Theorem B then shows that $\f$ has minimal singularities, and we obtain as a special case:

\begin{thmD}
{\it  Let $X$ be a smooth projective variety of general type. Then $X$ admits a unique singular K{\"a}hler-Einstein volume form of total mass equal to $\vol(K_X)$. In other words the canonical bundle $K_X$ can be endowed with a unique non-negatively curved metric $e^{-\phi_{KE}}$ whose curvature current $dd^c\phi_{KE}$ satisfies
\begin{equation}\label{equ:KEintro}\langle(dd^c\phi_{KE})^n\rangle=e^{\phi_{KE}}
\end{equation}
and such that
\begin{equation}\label{equ:massintro}\int_X e^{\phi_{KE}}=\vol(K_X).
\end{equation}
The weight $\phi_{KE}$ furthermore has minimal singularities. }
\end{thmD}

\noindent Since the canonical ring $R(K_X)= \oplus_{k \geq 0} H^0(k K_X)$ is now known to be finitely generated
~\cite{BCHM}, this result can be obtained as a consequence of~\cite{EGZ1} by passing to the canonical 
model of $X$. But one of the points of the proof presented here is to avoid the use of the difficult result of~\cite{BCHM}. 

The existence of $\phi_{KE}$ satisfying (\ref{equ:KEintro}) and (\ref{equ:massintro}) was also recently obtained by J.~Song and G.~Tian in~\cite{ST}, building on a previous approach of H.~Tsuji~\cite{Ts06}. It is also shown in~\cite{ST} that $\phi_{KE}$ is an \emph{analytic Zariski decomposition} (AZD for short) in the sense of Tsuji, which means that every pluricanonical section $\sigma\in H^0(mK_X)$ is $L^\infty$ with respect to the metric induced by $\phi_{KE}$. The main new information we add is that $\phi_{KE}$ actually has minimal singularities, which is strictly stronger than being an AZD for general big line bundles (cf.~Proposition~\ref{pro:azd}).  

\begin{ackn}
We would like to thank J.-P.Demailly and R.Berman for several useful discussions. 
\end{ackn}

\section{Non-pluripolar products of closed positive currents}

In this section $X$ denotes an arbitrary $n$-dimensional complex manifold unless otherwise specified.

\subsection{Plurifine topology}
The \emph{plurifine topology} on $X$ is defined as the coarsest topology with respect to which all psh functions $u$ on all open subsets of $X$ become continuous (cf.~\cite{BT2}). Note that the plurifine topology on $X$ is always strictly finer than the ordinary one. 

Since psh functions are upper semi-continuous, subsets of the form 
$$V\cap\{u>0\}$$
with $V\subset X$ open and $u$ psh on $V$ obviously form a basis for the plurifine topology, and $u$ can furthermore be assumed to be locally bounded by maxing it with $0$. 

When $X$ is furthermore compact and K{\"a}hler, a bounded local psh function defined on an open subset $V$ of $X$ can be extended to a quasi-psh function on $X$ (possibly after shrinking $V$ a little bit, see ~\cite{GZ1}), and it follows that the plurifine topology on $X$ can alternatively be described as the coarsest topology with respect to which all bounded and quasi-psh functions $\f$ on $X$ become continuous. It therefore admits sets of the form $V\cap\{\f>0\}$ with $V\subset X$ open and $\f$ quasi-psh and bounded on $X$ as a basis. 

\subsection{Local non-pluripolar products of currents}
Let $u_1,...,u_p$ be psh functions on $X$. If the $u_j$'s are locally bounded, then the fundamental work of Bedford-Taylor \cite{BT1} enables to define 
$$
dd^cu_1\wedge...\wedge dd^cu_p
$$ 
on $X$ as a closed positive $(p,p)$-current. The wedge product only depends on the closed positive $(1,1)$-currents $dd^c u_j$, and not on the specific choice of the potentials $u_j$. 

A very important property of this construction is that it is \emph{local in the plurifine topology}, in the following sense. If $u_j$, $v_j$ are locally bounded psh functions on $X$ and $u_j=v_j$ (pointwise) on a plurifine open subset $O$ of $U$, then 
$${\bf 1}_O dd^c u_1\wedge...\wedge dd^c u_p={\bf 1}_O dd^c v_1\wedge...\wedge dd^c v_p.$$
This is indeed an obvious generalization of Corollary 4.3 of~\cite{BT2}. 

In the case of a possibly unbounded psh function $u$, Bedford-Taylor have observed (cf.~p.236 of~\cite{BT2}) that it is always possible to define the \emph{non-pluripolar} part of the would-be positive measure $(dd^cu)^n$ as a Borel measure. The main issue however is that this measure is not going to be locally finite in general (see Example~\ref{ex:kisel} below). More generally, suppose we are trying to associate to any $p$-tuple $u_1,...,u_p$ of psh functions on $X$ a positive (\emph{a priori} not necessarily closed) $(p,p)$-current 
$$\langle dd^c u_1\wedge...\wedge dd^cu_p\rangle$$
putting no mass on pluripolar subsets, in such a way that the construction is local in the plurifine topology in the above sense. Then we have no choice: since $u_j$ coincides with the locally bounded psh function $\max(u_j,-k)$ on the plurifine open subset 
$$O_k:=\bigcap_j\{u_j> - k\},$$ 
we must have
\begin{equation}\label{equ:nonpp} 
{\bf 1}_{O_k}\langle\bigwedge_j dd^cu_j\rangle={\bf 1}_{O_k}\bigwedge_jdd^c\max(u_j,-k),
\end{equation}
Note that the right-hand side is non-decreasing in $k$. This equation completely determines $\langle\bigwedge_j dd^c u_j\rangle$ since the latter is required not to put mass on the pluripolar set $X-\bigcup_k O_k=\{u=-\infty\}$. 

\begin{defi} 
{\it If $u_1,...,u_p$ are psh functions on the complex manifold $X$, we shall say that the non-pluripolar product $\langle\bigwedge_j dd^cu_j\rangle$ is well-defined on $X$ if for each compact subset $K$ of $X$ we have
\begin{equation}\label{equ:welldef}
\sup_k\int_{K\cap O_k}\omega^{n-p}\wedge\bigwedge_jdd^c\max(u_j,-k)<+\infty
\end{equation}
for all $k$. }
\end{defi}

Here $\omega$ is an auxiliary (strictly) positive $(1,1)$-form on $X$ with respect to which masses are being measured, the condition being of course independent of $\omega$. When (\ref{equ:welldef}) is satisfied, equation (\ref{equ:nonpp}) indeed defines a positive $(p,p)$-current $\langle\bigwedge_j dd^cu_j\rangle$ on $X$. We will show below that it is automatically \emph{closed} (Theorem~\ref{thm:closed}). 

Condition (\ref{equ:welldef}) is always satisfied when $p=1$, and in fact it is not difficult to show that
$$\langle dd^c u\rangle={\bf 1}_{\{u>-\infty\}}dd^c u.$$
There are however examples where non-pluripolar products are not well-defined as soon as $p\ge 2$.
This is most easily understood in the following situation.

\begin{defi}\label{defi:small} 
{\it A psh function $u$ on $X$ will be said to have \emph{small unbounded locus} if there exists a (locally) complete pluripolar closed subset $A$ of $X$ outside which $u$ is locally bounded.}
\end{defi}

Assume that $u_1,...,u_p$ have small unbounded locus, and let $A$ be closed complete pluripolar such that each $u_j$ is locally bounded outside $A$ (recall that complete pluripolar subsets are stable under finite unions). Then $\langle\bigwedge_j dd^cu_j\rangle$ is well-defined iff the Bedford-Taylor product $\bigwedge_j dd^cu_j$, which is defined on the open subset $X-A$, has locally finite mass near each point of $A$. In that case $\langle\bigwedge_j dd^cu_j\rangle$ is nothing but the trivial extension of $\bigwedge_j dd^cu_j$ to $X$. 

\begin{ex}\label{ex:kisel} 
{\it Consider Kiselman's example (see ~\cite{Ki})
$$u(x,y):=(1-|x|^2)(-\log|y|)^{1/2}$$
for $(x,y)\in\C^2$ near $0$. The function $u$ is psh near $0$, it is smooth outside the $y=0$ axis, but the smooth measure $(dd^cu)^2$, defined outside $y=0$, is not locally finite near any point of $y=0$. This means that the positive Borel measure $\langle (dd^c u)^2\rangle$ is not locally finite. }
\end{ex}

We now collect some basic properties of non-pluripolar products.
\begin{pro}\label{pro:basic} Let $u_1,...,u_p$ be psh functions on $X$.
\begin{itemize} 
\item The operator $$(u_1,...,u_p)\mapsto\langle\bigwedge_j dd^c u_j\rangle$$
is local in the plurifine topology whenever well-defined.
\item The current $\langle\bigwedge_j dd^c u_j\rangle$ and the fact that it is well-defined both only depend on the currents $dd^cu_j$, not on the specific potentials $u_j$. 
\item Non-pluripolar products, which are obviously symmetric, are also multilinear in the following sense: if $v_1$ is a further psh function, then 
$$\langle(dd^cu_1+dd^cv_1)\wedge\bigwedge_{j\ge 2}dd^cu_j\rangle=\langle dd^c u_1\wedge\bigwedge_{j\ge 2}dd^cu_j\rangle+\langle dd^c v_1\wedge\bigwedge_{j\ge 2}dd^cu_j\rangle$$
in the sense that the left-hand side is well-defined iff both terms in the right-hand side are, and equality holds in that case. 
\end{itemize}
\end{pro}
\begin{proof} Let us prove the first point. If $u_j$ and $v_j$ are psh functions on $X$ such that $u_j=v_j$ on a given plurifine open subset $O$, then the locally bounded psh functions $\max(u_j,-k)$ and $\max(v_j,-k)$ also coincide on $O$. If we set 
$$E_k:=\bigcap_j\{u_j>-k\}\cap\bigcap_j\{v_j>-k\},$$
then we infer 
$${\bf 1}_{O\cap E_k}\bigwedge_j dd^c\max(u_j,-k)={\bf 1}_{O\cap E_k}\bigwedge_j dd^c\max(v_j,-k),$$
hence in the limit
$${\bf 1}_O\langle\bigwedge_j dd^cu_j\rangle={\bf 1}_O\langle\bigwedge_j dd^cv_j\rangle$$
as desired by Lemma~\ref{lem:useful} below. 

We now prove the second point. Let $w_j$ be pluriharmonic on $X$, and let $K$ be a compact subset. We can find $C>0$ such that $w_j\le C$ on an open neighborhood $V$ of $K$. On the plurifine open subset
$$O_k:=\bigcap_j\{u_j+w_j>-k\}\cap V\subset\bigcap_j\{u_j>-k-C\}\cap V$$
the following locally bounded psh functions coincide:
$$\max(u_j+w_j,-k)=\max(u_j,-w_j-k)+w_j=\max(u_j,-k-C)+w_j.$$
Since $dd^c w_j=0$, it follows that 
$$\int_{O_k}\omega^{n-p}\wedge\bigwedge_j dd^c\max(u_j+w_j,-k)=\int_{O_k}\omega^{n-p}\wedge\bigwedge_j dd^c\max(u_j,-k-C)$$
$$\le\int_{\bigcap_j\{u_j>-k-C\}\cap V}\omega^{n-p}\wedge\bigwedge_j dd^c\max(u_j,-k-C)$$
which is uniformly bounded in $k$ by assumption, and the second point is proved.

The proof of the last point is similarly easy but tedious and will be left to the reader.
\end{proof}

\begin{lem}\label{lem:useful} Assume that the non-pluripolar product $\langle\bigwedge_j dd^c u_j\rangle$ is well-defined. Then for every sequence of Borel subsets $E_k$ such that 
$$E_k\subset\bigcap_j\{u_i>-k\}$$
and $X-\bigcup_k E_k$ is pluripolar, we have
$$\lim_{k\to\infty}{\bf 1}_{E_k}\bigwedge_j dd^c\max(u_j,-k)\rangle=\langle\bigwedge_j dd^cu_j\rangle$$
against all bounded measurable functions. 
\end{lem}
\begin{proof} This follows by dominated convergence from
$$({\bf 1}_{\cap_j\{u_j>-k\}}-{\bf 1}_{E_k})\bigwedge_j dd^c\max(u_j,-k)\rangle\le(1-{\bf 1}_{E_k})\langle\bigwedge_j dd^cu_j\rangle.$$
\end{proof}

A crucial point for what follows is that non-pluripolar products of globally defined currents are \emph{always} well-defined on compact K{\"a}hler manifolds:
\begin{pro}\label{pro:finite} Let $T_1,...,T_p$ be closed positive $(1,1)$-currents on a compact K{\"a}hler manifold $X$. Then their non-pluripolar product $\langle T_1\wedge...\wedge T_p\rangle$ is well-defined. 
 \end{pro}
\begin{proof} Let $\omega$ be a K{\"a}hler form on $X$. In view of the third point of Proposition~\ref{pro:basic}, upon adding a large multiple of $\omega$ to the $T_j$'s we may assume that their cohomology classes are K{\"a}hler classes. We can thus find K{\"a}hler forms $\omega_j$ and $\omega_j$-psh functions $\f_j$ such that $T_j=\omega_j+dd^c\f_j$. Let $U$ be a small open subset of $X$ on which $\omega_j=dd^c\psi_j$, where $\psi_j\le 0$ is a smooth psh function on $U$, so that $T_j=dd^cu_j$ on $U$ with $u_j:=\psi_j+\f_j$. The bounded psh functions on $U$
$$\psi_j+\max(\f_j,-k)$$
and $$\max(u_j,-k)$$
coincide on the plurifine open subset $\{u_j>-k\}\subset\{\f_j>-k\}$, thus we have
$$\int_{\bigcap_j\{u_j>-k\}}\omega^{n-p}\wedge\bigwedge_j dd^c\max(u_j,-k)$$
$$=\int_{\bigcap_j\{u_j>-k\}}\omega^{n-p}\wedge\bigwedge_j(\omega_j+dd^c\max(\f_j,-k))$$
$$\le\int_X\omega^{n-p}\wedge\bigwedge_j(\omega_j+dd^c\max(\f_j,-k)).$$
But the latter integral is computed in cohomology, hence independent of $k$, and this shows that (\ref{equ:welldef}) is satisfied on $U$, qed. 
\end{proof}

\begin{rem} The same property extends to the case where $X$ is a compact complex manifold in the Fujiki class, that is bimeromorphic to a compact K{\"a}hler manifold. Indeed there exists in that case a modification $\mu:X'\to X$ with $X'$ compact K{\"a}hler. Since $\mu$ is an isomorphism outside closed analytic (hence pluripolar) subsets, it is easy to deduce that $\langle T_1\wedge...\wedge T_p\rangle$ is well-defined on $X$ from the fact that $\langle\mu^*T_1\wedge...\wedge\mu^*T_p\rangle$ is well-defined on $X'$, and in fact
$$\langle T_1\wedge...\wedge T_p\rangle=\mu_*\langle\mu^*T_1\wedge...\wedge\mu^*T_p\rangle.$$
On the other hand it seems to be unknown whether finiteness of non-pluripolar products holds on arbitrary compact complex manifolds. 
\end{rem}

Building on the proof of the Skoda-El Mir extension theorem, we will now prove the less trivial closedness property of non-pluripolar products.
\begin{thm}\label{thm:closed} Let $T_1,...,T_p$ be closed positive $(1,1)$-currents on a complex manifold $X$ whose non-pluripolar product is well-defined. Then the positive $(p,p)$-current $\langle T_1\wedge...\wedge T_p\rangle$ is closed. 
\end{thm}
\begin{proof} The result is of course local, and we can assume that $X$ is a small neighborhood of $0\in\C^n$ . The proof will follow rather closely Sibony's exposition of the Skoda-El Mir theorem \cite{Sib}, cf.~also \cite[pp.~159-161]{Dem0}. 

Let $u_j\le 0$ be a local potential of $T_j$ near $0\in\C^n$, and for each $k$ consider the closed positive current of bidimension $(1,1)$ 
$$\Theta_k:=\rho\wedge\bigwedge_j dd^c\max(u_j,-k)$$
and the plurifine open subset 
$$O_k:=\bigcap_j\{u_j>-k\}$$
so that ${\bf 1}_{O_k}\Theta_k$ converges towards 
$$\rho\wedge\langle T_1\wedge...\wedge T_p\rangle$$ 
by (\ref{equ:nonpp}). Here $\rho$ is a positive $(n-p-1,n-p-1)$-form with constant coefficients, so that $\Theta_k$ has bidimension $(1,1)$. It is enough to show that
$$\lim_{k\to\infty}d\left({\bf 1}_{O_k}\Theta_k\right)=0$$
for any choice of such a form $\rho$.  

Let also $u:=\sum_j u_j$, so that $u\le-k$ outside $O_k$, and set 
$$w_k:=\chi(e^{u/k}),$$
where $\chi(t)$ is a smooth convex and non-decreasing function of $t\in\R$ such that $\chi(t)=0$ for $t\le 1/2$ and $\chi(1)=1$. 
We thus see that $0\le w_k\le 1$ is a non-decreasing sequence of bounded psh functions defined near $0\in\C^n$ with $w_k=0$ outside $O_k$
and $w_k\to 1$ pointwise outside the pluripolar set $\{u=-\infty\}$. Finally let $0\le\theta(t)\le 1$ be a smooth non-decreasing function of $t\in\R$ such that $\theta(t)=0$ for $t\le 1/2$ and $\theta\equiv 1$ near $t=1$. The functions $\theta(w_k)$ are bounded, non-decreasing in $k$ and we have
$$\theta(w_k)\le{\bf 1}_{O_k}$$
since $\theta(w_k)\le 1$ vanishes outside $O_k$. Note also that $\theta'(w_k)$ vanishes outside $O_k$, and converges to $0$ pointwise outside $\{u=-\infty\}$. 

Our goal is to show that
$$\lim_{k\to\infty}d\left({\bf 1}_{O_k}\Theta_k\right)=0.$$ 
But we have
$$0\le\left({\bf 1}_{O_k}-\theta(w_k)\right)\Theta_k\le\left(1-\theta(w_k)\right)\langle T_1\wedge...\wedge T_p\rangle,$$
and the latter converges to $0$ by dominated convergence since $\theta(w_k)\to 1$ pointwise outside the polar set of $u$, which is negligible for $\langle T_1\wedge...\wedge T_p\rangle$. It is thus equivalent to show that
$$\lim_{k\to\infty}d\left(\theta(w_k)\Theta_k\right)=0.$$
Since $w_k$ is a bounded psh function, Lemma~\ref{lem:chainrule} below shows that the chain rule applies, that is
$$d\left(\theta(w_k)\Theta_k\right)=\theta'(w_k)dw_k\wedge\Theta_k.$$
Recall that $dw_k\wedge\Theta_k$ has order $0$ by Bedford-Taylor, so that the right-hand side makes sense. Now let $\psi$ be a given smooth $1$-form compactly supported near $0$ and let $\tau\ge 0$ be a smooth cut-off function with $\tau\equiv 1$ on the support of $\psi$. The Cauchy-Schwarz inequality implies
$$\left|\int\theta'(w_k)\psi\wedge dw_k\wedge\Theta_k\right|^2\le\left(\int\tau dw_k\wedge d^c w_k\wedge\Theta_k\right)\left(\int\theta'(w_k)^2\psi\wedge\overline{\psi}\wedge\Theta_k\right).$$
But on the one hand we have 
$$2\int\tau dw_k\wedge d^cw_k\wedge\Theta_k\le\int\tau dd^cw_k^2\wedge\Theta_k$$
$$=\int w_k^2 dd^c\tau\wedge\Theta_k=\int_{O_k} w_k^2 dd^c\tau\wedge\Theta_k$$
since $w_k$ vanishes outside $O_k$, and the last integral is uniformly bounded since $0\le w_k\le 1$ and ${\bf 1}_{O_k}\Theta_k$ has uniformly bounded mass by (\ref{equ:welldef}). On the other hand we have 
$$\theta'(w_k)^2\Theta_k=\theta'(w_k)^2{\bf 1}_{O_k}\Theta_k\le\theta'(w_k)^2\langle T_1\wedge...\wedge T_p\rangle$$
since $\theta'(w_k)$ also vanishes outside $O_k$, and we conclude that
$$\lim_{k\to\infty}\int\theta'(w_k)^2\psi\wedge\overline{\psi}\wedge\Theta_k=0$$
by dominated convergence since $\theta'(w_k)\to 0$ pointwise ouside the polar set of $u$, which is negligible for $\langle T_1\wedge...\wedge T_p\rangle$. The proof is thus complete.
\end{proof}

\begin{lem}\label{lem:chainrule} Let $\Theta$ be a closed positive $(p,p)$-current on a complex manifold $X$, $f$ be a smooth function on $\R$ and $v$ be a bounded psh function. Then we have 
$$d(f(v)\Theta)=f'(v)dv\wedge\Theta.$$
\end{lem}
\begin{proof} This is a local result, and we can thus assume that $\Theta$ has bidimension $(1,1)$ by multiplying it by constant forms as above. The result is of course true when $v$ is smooth. As we shall see the result holds true in our case basically because $v$ belongs to the Sobolev space $L^2_1 (\Theta)$, in the sense that $dv\wedge d^c v\wedge\Theta$ is well-defined. Indeed the result is standard when $\Theta=[X]$, and proceeds by approximation. Here we let $v_k$ be a decreasing sequence of smooth psh functions converging pointwise to $v$. We then have $f(v_k)\theta\to f(v)\Theta$ by dominated convergence, thus it suffices to show that
$$\lim_{k\to\infty} f'(v_k)dv_k\wedge\Theta=f'(v)dv\wedge\Theta.$$
We write
$$f'(v_k)dv_k\wedge\Theta-f'(v)dv\wedge\Theta=\left(f'(v_k)-f'(v)\right)dv_k\wedge\Theta+f'(v)(dv_k-dv)\wedge\Theta.$$
Let $\psi$ be a test $1$-form and $\tau\ge 0$ be a smooth cut-off function with $\tau\equiv 1$ on the support of $\psi$.  Cauchy-Schwarz implies
$$\left|\int\left(f'(v_k)-f'(v)\right)\psi\wedge dv_k\wedge\Theta\right|^2$$
$$\le\left(\int\left(f'(v_k)-f'(v)\right)^2\psi\wedge\overline{\psi}\wedge\Theta\right)\left(\int\tau dv_k\wedge d^c v_k\wedge\Theta\right).$$
The second factor is bounded since $dv_k\wedge d^cv_k\wedge\Theta$ converges to $dv\wedge d^cv\wedge\Theta$ by Bedford-Taylor's monotone convergence theorem, and the first one converges to $0$ by dominated convergence. 
We similarly have
$$\left|\int f'(v)(dv_k-dv)\wedge\psi\wedge\Theta\right|^2$$
$$\le\left(\int f'(v)^2\psi\wedge\overline{\psi}\wedge\Theta\right)\left(\int\tau d(v_k-v)\wedge d^c(v_k-v)\wedge\Theta\right),$$
where now the first factor is bounded while the second one tends to $0$ by Bedford-Taylor once again, and the proof is complete.
\end{proof}

\begin{rem} 
{\it Injecting as above Lemma~\ref{lem:chainrule} in the proof of Skoda-El Mir's extension theorem presented in~\cite{Dem0} (p.159-161) shows that Skoda-El Mir's result remains true for complete pluripolar subsets that are not necessarily \emph{closed}, in the following sense: let $\Theta$ be a closed positive $(p,p)$-current and let $A$ be a complete pluripolar subset of $X$. Then ${\bf 1}_{X-A}\Theta$ and thus also ${\bf 1}_{A}\Theta$ are closed. }
\end{rem}

We conclude this section with a log-concavity property of non-pluripolar products. Let $T_1,...,T_n$ be closed positive $(1,1)$-currents with locally bounded potentials near $0\in\C^n$ and let $\mu$ be a positive measure. Suppose also given non-negative measurable functions $f_j$ such that
\begin{equation}\label{equ:ma_minore} T_j^n\ge f_j\mu,\,j=1,...n.
\end{equation}
Theorem 1.3 of~\cite{Din1} then implies that 
\begin{equation}\label{equ:BM} T_1\wedge...\wedge T_n\ge(f_1...f_n)^{1/n}\mu.
\end{equation}
It is in fact a standard variation of the Brunn-Minkowski inequality that (\ref{equ:ma_minore}) implies (\ref{equ:BM}) when the whole data is smooth, and~\cite{Din1} reduces to this case by an adequate regularization process. As an easy consequence of Dinew's result, we get the following version for non-pluripolar products.

\begin{pro}\label{pro:logconcave} Let $T_1,...,T_n$ be closed positive $(1,1)$-currents, let $\mu$ be a positive measure and assume given for each $j=1,...,n$ a non-negative measurable function $f_j$ such that
$$\langle T_j^n\rangle\ge f_j\mu.$$
Then we have
$$\langle T_1\wedge...\wedge T_n\rangle\ge(f_1...f_n)^{\frac 1 n}\mu.$$
\end{pro}
Here non-pluripolar products are considered as (possibly locally infinite) positive Borel measures. \begin{proof} Write $T_j=dd^c u_j$ with $u_j$ psh and consider the plurifine open subset
$$O_k:=\bigcap_j\{u_j>-k\}.$$
Since non-pluripolar products are local in plurifine topology, we get for each $j$ and $k$
$$(dd^c\max(u_j,-k))^n\ge{\bf 1}_{O_k}f_j\mu,$$
hence 
$$dd^c\max(u_1,-k)\wedge...\wedge dd^c\max(u_n,-k)\ge{\bf 1}_{O_k}(f_1...f_n)^{\frac 1 n}\mu$$
by Dinew's result, or equivalently by plurifine locality
$$\langle dd^c u_1\wedge...\wedge dd^c u_n\rangle\ge{\bf 1}_{O_k}(f_1...f_n)^{\frac 1 n}\mu.$$
Since this is valid for all $k$ and the complement of  $\bigcup_k O_k$ is pluripolar, the result follows. 
\end{proof}

From now on we will work in the global setting where $X$ is a compact K{\"a}hler manifold.

\subsection{Positive cohomology classes}
Let $X$ be a compact K{\"a}hler manifold and let $\a\in H^{1,1}(X,\R)$ be a real $(1,1)$-cohomology  class.

Recall that $\a$ is said to be \emph{pseudo-effective} (\emph{psef} for short) if it can be represented by a closed positive $(1,1)$-current $T$. If $\theta$ is a given smooth representative of $\a$, then any such current can be written as $T=\theta+dd^c\f$, where $\f$ is thus a $\theta$-psh function by definition, and will sometimes be referred to as a \emph{global potential} of $T$. Global potentials only depend on the choice of $\theta$ up to a smooth function on $X$. 

On the other hand, $\a$ is \emph{nef} if it lies in the closure of the K{\"a}hler cone. Such a class can thus be represented by \emph{smooth} forms with arbitrarily small negative parts, and a compactness argument then shows that $\a$ contains a positive current, i.e.~nef implies psef. Note however that the negative part of smooth representatives cannot be taken to be $0$ in general. In fact, Demailly, Peternell and Schneider have shown in~\cite{DPS} that a classical construction of Serre yields an example of a smooth curve $C$ on a projective surface whose cohomology class is nef but contains only one positive current, to wit (the integration current on) $C$ itself. 

The set of all nef classes is by definition a closed convex cone in $H^{1,1}(X,\R)$, whose interior is none but the K{\"a}hler cone. The set of all psef classes also forms a closed convex cone, and its interior is by definition the set of all \emph{big} cohomology classes. In other words, a class $\a$ is big iff it can be represented by a \emph{strictly positive} current, i.e.~a closed current $T$ that dominates some (small enough) smooth strictly positive form on $X$.

\subsection{Comparison of singularities} 
If $T$ and $T'$ are two closed positive currents on $X$, then $T$ is said to be \emph{more singular} than $T'$ if their global potentials satisfy $\f\le\f'+O(1)$. By the \emph{singularity type} of $T$, we will mean its equivalence class with respect to the above notion of comparison of singularities. 

A positive current $T$ in a given psef class $\a$ is now said to have \emph{minimal singularities} (inside its cohomology class) if it is less singular than any other positive current in $\a$, and its $\theta$-psh potentials $\f$ will correspondingly be said to have minimal singularities. Such $\theta$-psh functions with minimal singularities always exist, as was observed by Demailly. Indeed, the upper envelope 
$$V_\theta:=\sup\left\{ \f\,\,\theta\text{-psh}, \f\le 0\text{ on } X \right \}$$ 
of all non-positive $\theta$-psh functions obviously yields a positive current 
$$\theta+dd^cV_\theta$$ 
with minimal singularities (note that $V_\theta$ is usc since $V_\theta^*$ is a candidate in the envelope). 

We stress that currents with minimal singularities in a given class $\a$ are in general far from unique. Indeed currents with minimal singularities in a class $\a$ admitting a smooth non-negative representative $\theta\ge 0$ (for instance a K{\"a}hler class) are exactly those with bounded potentials. 

Currents with minimal singularities are stable under pull-back:

\begin{pro}\label{pro:tire_min} 
Let $\pi:Y\to X$ be a surjective morphism between compact K{\"a}hler manifolds and $\theta$ be a smooth closed $(1,1)$-form on $X$. If $\f$ is a $\theta$-psh function with minimal singularities on $X$, then the $\pi^*\theta$-psh function $\f\circ\pi$ also has minimal singularities.
\end{pro}

 \begin{proof} 
 Let $\psi$ be a $\pi^*\theta$-psh function. For $x\in X$ a regular value of $\pi$ we set
$$\tau(x):=\sup_{\pi(y)=x}\psi(y).$$
It is standard to show that $\tau$ uniquely extends to a $\theta$-psh function on $X$, so that $\tau\le\f+O(1)$. But this clearly implies $\psi\le\f\circ\pi+O(1)$ as was to be shown.  
 \end{proof}
 
If $\pi$ furthermore has connected fibers then any $\pi^*\theta$-psh function on $Y$ is of the form $\f \circ \pi,$ where $\f$ is a $\theta$-psh function on $X$.

\medskip
 
A positive current $T=\theta+dd^c\f$ and its global potential $\f$ are said to have \emph{analytic singularities} if there exists $c>0$ such that (locally on $X$),
$$
\f=\frac{c}{2}\log\sum_j|f_j|^2+u,
$$
where $u$ is smooth and $f_1,...f_N$ are local holomorphic functions. The coherent ideal sheaf $\cI$ locally generated by these functions (in fact, its integral closure) is then globally defined, and the singularity type of $T$ is thus encoded by the data formally denoted by $\cI^c$. 

Demailly's  fundamental regularization theorem \cite{Dem1} states that a given $\theta$-psh function $\f$ can be approximated from above by a sequence $\f_k$ of $(\theta+\e_k\omega)$-psh functions with analytic singularities, $\omega$ denoting some auxiliary K{\"a}hler form. In fact the singularity type of $\f_k$ is decribed by the $k$-th root of the multiplier ideal sheaf of $k\f$. 

This result implies in particular that a big class $\a$ always contains strictly positive currents with analytic singularities. It follows that there exists a Zariski open subset $\Omega$ of $X$ on which global potentials of currents with minimal singularities in $\a$ are all locally bounded. 
The following definition is extracted from~\cite{Bou2}:

\begin{defi}
{\it If $\a$ is a big class, we define its \emph{ample locus} $\Amp(\a)$ as the set of points $x\in X$ such that there exists a strictly positive current $T\in\a$ with analytic singularities and smooth around $x$.}
\end{defi}

The ample locus $\Amp(\a)$ is a Zariski open subset by definition, and it is nonempty thanks to Demaillly's regularization result. In fact it is shown in \cite{Bou2} that there exists a strictly positive current $T\in\a$ with analytic singularities whose smooth locus is precisely $\Amp(\a)$. Note that $\Amp(\a)$ coincides with the complement of the so-called \emph{augmented base locus} $\mathbb{B}_+(\a)$ (see~\cite{ELMNP06}) when $\a=c_1(L)$ is the first Chern class of a big line bundle $L$. 

\subsection{Global currents with small unbounded base locus}
Let $X$ be a compact K{\"a}hler manifold, and let $T_1,...,T_p$ be closed positive $(1,1)$-currents on $X$ with small unbounded locus (Definition~\ref{defi:small}). We can then find a closed complete pluripolar subset $A$ such that each $T_j$ has locally bounded potentials on $X-A$.  The content of Proposition~\ref{pro:finite} in that case is that the Bedford-Taylor product $T_1\wedge...\wedge T_p$, which is well-defined on $X-A$, has finite total mass on $X-A$:
$$
\int_{X - A} T_1 \wedge \cdots \wedge T_p \wedge \omega^{n - p} < + \infty,
$$ 
for any K{\"a}hler form $\omega$ on $X$.

We now establish a somewhat technical-looking integration-by-parts theorem in this context that will be crucial in order to establish the basic properties of weighted energy functionals in Proposition~\ref{pro:monotone}. 

\begin{thm}\label{thm:intbyparts} 
Let $A\subset X$ be a closed complete pluripolar subset, and let $\Theta$ be a closed positive current on $X$ of bidimension $(1,1)$. Let $\f_i$ and $\psi_i$, $i=1,2$ be quasi-psh functions on $X$ that are locally bounded on $X-A$. If $u:=\f_1-\f_2$ and $v:=\psi_1-\psi_2$ are globally bounded on $X$, then 
$$
\int_{X-A} udd^cv\wedge\Theta=\int_{X-A} vdd^cu\wedge\Theta=-\int_{X-A} dv\wedge d^c u\wedge\Theta.
$$
\end{thm} 

Note that the (signed) measure $dd^c\f_i\wedge\Theta$ defined on $X-A$ has finite total mass thus so does $dd^cu\wedge\Theta$ and $\int_{X-A} vdd^cu\wedge\Theta$ is therefore well-defined since $v$ is bounded (and defined everywhere). The last integral is also well defined 
by Lemma~\ref{lem:finitemass} below and the Cauchy-Schwarz inequality. 

\begin{proof} 
As is well-known, if $\f\ge 0$ is a bounded psh function, the identity 
$$dd^c\f^2=2d\f\wedge d^c\f+2\f dd^c\f$$
enables to define $d\f\wedge d^c\f\wedge\Theta$ as a positive measure, which shows that $d^c\f\wedge\Theta:=d^c(\f\Theta)$ is a current of order $0$ by the Cauchy-Schwarz inequality. 
By linearity we can therefore make sense of $vd^cu\wedge\Theta$ and $du\wedge d^cu\wedge\Theta$ as currents of order $0$ on $X-A$. By Lemma~\ref{lem:finitemass} below the positive measure $du\wedge d^c u\wedge\Theta$ on $X-A$ has finite total mass, thus so does the current of order $0$ $vd^c u\wedge\Theta$ by Cauchy-Schwarz.

We claim that 
\begin{equation}\label{equ:leib}
 d[v d^c u\wedge\Theta]=dv\wedge d^c u\wedge\Theta+v dd^c u\wedge\Theta\ \text{on} \ X-A.
\end{equation}
 To show this, we argue locally. We can then assume that $u$ and $v$ are locally bounded psh functions by bilinearity. Let $u_k$ and $v_k$ be smooth psh functions decreasing towards $u$ and $v$ respectively. Since $d^c u\wedge\Theta$ has order $0$, we see that $v_k d^cu\wedge\Theta$ converges to $v d^c u\wedge\Theta$ by monotone convergence. Since the right-hand side of the desired formula is continuous along decreasing sequences of bounded psh functions by Bedford-Taylor's theorem, we can thus assume that $v$ is smooth and psh. Now $u_k\Theta$ converges to $u\Theta$, thus $v d^c u_k\wedge\Theta\to v d^c u_k\wedge\Theta$, and we are done by another application of Bedford-Taylor's theorem to the right-hand side.   

In particular (\ref{equ:leib}) shows that $d[vd^cu\wedge\Theta]$ is a current of order $0$ on $X-A$. 
Let us somewhat abusively denote by ${\bf 1}_{X-A} vd^cu\wedge\Theta$ and ${\bf 1}_{X-A} d[vd^cu\wedge\Theta]$ the trivial extensions to $X$. We are going to show that
\begin{equation}\label{equ:residu}
d[{\bf 1}_{X-A} vd^c u\wedge\Theta]={\bf 1}_{X-A} d[v d^c u\wedge\Theta]
\end{equation}
as currents on $X$. By Stokes' theorem we will thus have 
$$\int_{X-A} d[v d^cu\wedge\Theta]=0,$$
and the desired formulae will follow by (\ref{equ:leib}). 

 The proof of (\ref{equ:residu}) will again follow rather closely Sibony's proof of the Skoda-El Mir extension theorem \cite{S,Dem0}. Note that (\ref{equ:residu}) is a local equality, and that it clearly holds on $X-A$. We will thus focus on a fixed small enough neighbourhood $U$ of a given point $0\in A$. Since $A$ is complete pluripolar around $0$, we may by definition find a psh function $\tau$ on $U$ such that 
 $\{\tau=-\infty\}=A$ near $0$. If we pick a smooth non-decreasing convex function $\chi$ on $\R$ such that $\chi(t)=0$ for $t\le 1/2$ and $\chi(1)=1$ and finally set 
$$w_k:=\chi(e^{\tau/k}),$$
then $0\le w_k\le 1$ is a bounded psh function on $U$ vanishing on $A$, and the sequence $w_k$ increases to $1$ pointwise on $U-A$ as $k\to\infty$. Now choose a smooth non-decreasing function $0\le\theta(t)\le 1$, $t\in\R$ such that $\theta\equiv 0$ near $0$ and $\theta\equiv 1$ near $1$ (we of course cannot require $\theta$ to be convex here). It then follows that $\theta(w_k)$ vanishes near $A$ and $\theta(w_k)\to 1$ pointwise on $U-A$, whereas $\theta'(w_k)\to 0$ pointwise on $U$. As a consequence we get
$$\theta(w_k)vd^c u\wedge\Theta\to{\bf 1}_{U-A} vd^c u\wedge\Theta$$
and
$$
\theta(w_k)d[v d^c u\wedge\Theta]\to{\bf 1}_{U-A} d[v d^c u\wedge\Theta]
$$
on $U$ as $k\to\infty$. Since 
$$
d[\theta(w_k)vd^c u\wedge\Theta]=\theta'(w_k)vdw_k\wedge d^cu\wedge\Theta
+\theta(w_k)d[vd^cu\wedge\Theta]
$$
by Lemma~\ref{lem:chainrule}, we are thus reduced to showing that 
$$\theta'(w_k)vdw_k\wedge d^c u\wedge\Theta\to 0$$
on $U$ as $k\to\infty$. But if $\psi\ge 0$ is a given smooth function on $U$ with compact support, Cauchy-Schwarz implies
\begin{eqnarray*}
 && \left|\int\psi\theta'(w_k)vdw_k\wedge d^cu\wedge\Theta\right|^2 \\
&\le& \left(\int\psi dw_k\wedge d^cw_k\wedge\theta\right) \left(\int\psi \theta'(w_k)^2v^2du\wedge d^c u\wedge\Theta\right),
\end{eqnarray*}
and one concludes since 
\begin{eqnarray*}
2\int\psi dw_k\wedge d^cw_k\wedge\Theta & \le & \int\psi dd^cw_k^2\wedge\Theta \\
&=& \int w_k^2 dd^c\psi\wedge\Theta
\end{eqnarray*} 
is bounded whereas 
$$\int\psi \theta'(w_k)^2v^2du\wedge d^c u\wedge\Theta\to 0$$ by dominated convergence since $\theta'(w_k)^2\to 0$ pointwise. 
\end{proof}

\begin{lem}\label{lem:finitemass} Let $\f_1$ and $\f_2$ be two quasi-psh functions on $X$ that are locally bounded on an open subset $U$. If the difference $u:=\f_1-\f_2$ is globally bounded on $X$, then 
$$\int_U du\wedge d^c u\wedge\Theta<+\infty.$$ 
\end{lem} 

\begin{proof} By Demailly's approximation theorem~\cite{Dem1} we can find decreasing sequences of smooth functions $\f_i^{(k)}$ decreasing towards $\f_i$ as $k\to\infty$ and such that $dd^c\f_i^{(k)}\ge-\omega$ for some fixed large enough K{\"a}hler form $\omega$ (see also~\cite{BK}).  Since $u=\f_1-\f_2$ is bounded by assumption, it follows that $u_k:=\f_1^{(k)}-\f_2^{(k)}$ can be taken to be uniformly bounded with respect to $k$. We can thus assume that $0\le u_k\le C$ for all $k$. 

Now the formula $dd^c v^2=2vdd^cv+2dv\wedge d^c v$ and Bedford-Taylor's continuity theorem implies by polarization identities that
$$\lim_kd\f_i^{(k)}\wedge d^c\f_j^{(k)}\wedge\Theta=d\f_i\wedge d^c\f_j\wedge\Theta$$ 
weakly on $U$ for any two $i,j$, hence also 
$$du_k\wedge d^c u_k\wedge\Theta\to du\wedge d^c u\wedge\Theta$$
weakly on $U$. We thus get
$$\int_U du\wedge d^cu\wedge\Theta\le\liminf_k\int_X du_k\wedge d^c u_k\wedge\Theta,$$
and our goal is to show that the right-hand integrals are uniformly bounded. 
Now Stokes theorem yields
$$\int_X du_k\wedge d^c u_k\wedge\Theta=-\int_X u_kdd^cu_k\wedge\Theta$$
$$=\int_X u_k dd^c\f_2^{(k)}\wedge\Theta-\int_X u_k dd^c\f_1^{(k)}\wedge\Theta.$$
For each $i=1,2$ we have
$$0\le u_k(\omega+dd^c\f_i^{(k)})\le C(\omega+dd^c\f_i^{(k)})$$
hence $$0\le\int_Xu_k(\omega+dd^c\f_i^{(k)})\wedge\Theta\le C\int_X\omega\wedge\Theta$$
so that 
$$-C\int_X\omega\wedge\Theta\le\int_X u_k dd^c\f_i^{(k)}\wedge\Theta\le C\int_X\omega\wedge\Theta$$
for all $k$, and the uniform boundedness of $\int_Xdu_k\wedge d^cu_k\wedge\Theta$ follows.  
\end{proof} 

The following crucial result shows that non-pluripolar masses are basically non-increasing with singularities, at least for currents with small unbounded locus.  

\begin{thm}
\label{thm:comp} For $j=1,...,p$, let $T_j$ and $T'_j$ be two cohomologous closed positive $(1,1)$-currents with small unbounded locus, and assume also that $T_j$ is less singular than $T_j'$. Then the cohomology classes of their non-pluripolar products satisfy
$$\{\langle T_1\wedge...\wedge T_p\rangle\}\ge\{\langle T'_1\wedge...\wedge T'_p\rangle\}$$
in $H^{p,p}(X,\R)$, with $\ge$ meaning that the difference is pseudo-effective, i.e.~representable by a closed positive $(p,p)$-current. 
\end{thm}

The statement of course makes sense for arbitrary closed positive $(1,1)$-currents, and we certainly expect it to hold true in general, but we are only able for the moment to prove it in the special case of currents with small unbounded locus. 

\begin{proof} By duality, this is equivalent to showing that
$$
\int_{X-A} T_1\wedge...\wedge T_p\wedge\tau\ge\int_{X-A} T_1'\wedge...\wedge T'_p\wedge\tau
$$
for every positive smooth $dd^c$-closed $(n-p,n-p)$-form $\tau$, where $A$ a closed complete pluripolar set outside of which all currents $T_i, T'_i$ have locally bounded potentials. 
Replacing successively $T_{i}$ by $T_{i}'$, we can assume that
$T_{i}=T_{i}'$ for $i>1$. The bidimension $(1,1)$ current  
$$\Theta:=\langle T_{2}\wedge...\wedge T_{p}\rangle\wedge\tau$$ 
is $dd^c$-closed and positive. Writing $T_{1}=\theta+dd^{c}\f$ and $T_{1}'=\theta+dd^{c}\f'$ with $\theta$ a smooth form in their common cohomology class, we are reduced to showing that 
$$\int_{X-A}dd^c\f\wedge\Theta\geq\int_{X-A}dd^c\f'\wedge\Theta$$
if $\f\ge\f'$ are quasi-psh functions on $X$ locally bounded on $X-A$ and $\Theta$ is a $dd^c$-closed positive bidimension $(1,1)$ current on $X$. 
Let $\psi$ be a quasi-psh function such that $A=\{\psi=-\infty\}$. Replacing $\f'$ by $\f'+\e\psi$ and letting $\e\rightarrow0$ in the end, we may further assume that $\f'-\f\rightarrow-\infty$ near $A$. For each $k$ the function $\psi_k:=\max(\f',\f-k)$ thus coincides with $\f-k$ in some neighbourhood of $A$, so that
$$\int_{X-A}dd^{c}\psi_k\wedge\Theta=\int_{X-A}dd^{c}\f\wedge\Theta$$
by Stokes theorem, since (the trivial extension of) $dd^c(\f-\psi_k)\wedge\Theta$ is the exterior derivative of a current with compact support on $X-A$.  On the other hand 
 $$\lim_k dd^{c}\psi_k\wedge\Theta=dd^{c}\f'\wedge\Theta$$
weakly on $X-A$, and the result follows. 
\end{proof}

As a consequence we can introduce 

\begin{defi} 
{\it Let $\a_1,...\a_p\in H^{1,1}(X,\R)$ be big cohomology classes, and let $T_{i,\min}\in\a_i$ be a positive current with minimal singularities. Then the cohomology class of the non-pluripolar product $\langle T_{1,\min}\wedge \cdots \wedge T_{p,\min}\rangle$ is independent of the choice of $T_{i,\min}\in\a_i$ with minimal singularities. It will be denoted by 
$$\langle\a_1\cdots \a_p\rangle\in H^{p,p}(X,\R)$$
and called the \emph{positive product} of the $\a_i$. If $\a_1,...,\a_n$ are merely psef, we set
$$
\langle\a_1\cdots \a_p\rangle:=\lim_{\e\to 0}\langle(\a_1+\e\beta)\cdots (\a_p+\e\beta)\rangle
$$
where $\beta\in H^{1,1}(X,\R)$ is an arbitrary K{\"a}hler class. 

As a special case, given a big class $\a$ the positive number 
$$\vol(\a):=\langle \a^n\rangle$$
 is called the \emph{volume} of $\a$}
\end{defi} 

Some explanations are in order. The positive product $\langle\a_{1}\cdots \a_{p}\rangle$
is \emph{not} multi-linear since the fact that $T\in\a$ and $T'\in\a'$ both have minimal singularities (in their respective cohomology classes) doesn't imply that $T+T'$ has minimal singularities in $\a+\a'$. Positive products are however homogeneous and non-decreasing in each variable $\a_i$ thanks to Theorem~\ref{thm:comp} when the $\a_i$'s is big (with respect to the partial order on cohomology induced by positive currents). These two properties imply in a completely formal way (cf.~Proposition 2.9 in~\cite{BFJ}) that $\langle\a_1\cdot...\cdot\a_p\rangle$ depends \emph{continuously} on the $p$-tuple $(\a_{1},...,\a_{p})$ of \emph{big} classes. We can thus extend positive products to psef classes $\a_i$ as indicated, the existence of the limit and its independence on $\beta$ being easy to check by monotonicity of positive products. 
 
As an important special case, note that the positive product of \emph{nef} classes $\a_i$ coincides with their ordinary cup-product, that is
$$\langle\a_1\cdots \a_p\rangle=\a_1\cdots \a_p.$$

The volume $\vol(\a)$ considered here coincides (as it should!) with the one  introduced in~\cite{Bou1}. In fact: 

\begin{pro}\label{pro:singanal} Let $\a_1,...,\a_p$ be big cohomology classes. Then there exists sequences $T_j^{(k)}\in\a_j$ of strictly positive currents with analytic singularities such that
$$\lim_{k\to\infty}\{\langle T_1^{(k)}\wedge \cdots \wedge T_p^{(k)}\rangle\}=\langle\a_1\wedge...\wedge \a_p\rangle$$
in $H^{p,p}(X,\R)$. 
\end{pro}
\begin{proof} To keep notations simple we assume that $\a_1= \cdots =\a_p$. Let $T_{\min}\in\a$ be a positive current with minimal singularities,  let $T_+\in\a$ be a given strictly positive current with analytic singularities, and let $\omega$ be a K{\"a}hler form such that $T_+\ge\omega$. By Demailly's regularization theorem, there exists a sequence $S_k\in\a$ of currents with analytic singularities whose global potentials decrease to that of $T_{\min}$ and such that $S_k\ge-\e_k\omega$. Now let
$$T_k:=(1-\e_k)S_k+\e_k T_+,$$
so that $T_k\ge\e_k^2\omega$ is indeed strictly positive with analytic singularities. It is clear that $T_k$  converges to $T_{\min}$ as $k\to\infty$. By continuity of mixed Monge-Amp{\`e}re operators along decreasing sequences of locally bounded psh functions, we get
$$T_k^p\to T_{\min}^p$$
weakly on $X-A$ with $A$ as above. We infer that
$$\int_{X-A} T_{\min}^p\wedge\omega^{n-p}\le\liminf_{k\to\infty}\int_{X-A} T_k^p\wedge\omega^{n-p}.$$
On the other hand Theorem~\ref{thm:comp} yields
$$\int_{X-A} T_{\min}^p\wedge\omega^{n-p}\ge\int_{X-A} T_k^p\wedge\omega^{n-p},$$
since $T_{\min}$ has minimal singularities, and together this implies
$$\lim_k\langle T_k^p\rangle=\langle T_{\min}^p\rangle$$
as $(p,p)$-currents on $X$. The result follows.  
\end{proof}

The next result corresponds to Theorem 4.15 of~\cite{Bou1}, which generalizes Fujita's approximation theorem to arbitrary $(1,1)$-classes. We reprove it here for convenience of the reader. 

\begin{pro}\label{cor:approxzar} 
Let $\a$ be a big class on $X$. Then for each $\e>0$, there exists a modification $\pi:X'\to X$ with $X'$ a compact K{\"a}hler manifold, a K{\"a}hler class $\beta$ on $X'$ and an effective $\R$-divisor $E$ such that
\begin{itemize}
\item $\pi^*\a=\beta+\{E\}$ and
\item $\vol(\a)-\e\le\vol(\beta)\le\vol(\a).$
\end{itemize}
\end{pro}

\begin{proof} 
By Proposition~\ref{pro:singanal}, there exists a positive current $T\in\a$ with analytic singularities described by $\cI^c$ for some $c>0$ and some coherent ideal sheaf $\cI$ and a K{\"a}hler form $\omega$ such that $\int_X\langle T^n\rangle>\vol(\a)-\e$ and $T\ge\omega$. By Hironaka's theorem, there exists a finite sequence of blow-ups with smooth centres $\pi:X'\to X$ such that $\pi^{-1}\cI$ is a principal ideal sheaf. The Siu decomposition of $\pi^*T$ thus writes 
$$
\pi^*T=\theta+[D]
$$
where $D$ is an effective $\R$-divisor and $\theta\ge 0$ is a smooth form. It is also straightforward to see that $\int_X\langle T^n\rangle=\int_{X'}\theta^n$. The condition $T\ge\omega$ yields $\theta\ge\pi^*\omega$, but this is not quite enough to conclude that the class of $\theta$ is K{\"a}hler, in which case we would be done. However since $\pi$ is a finite composition of blow-ups with smooth centers, it is well-known that there exists a $\pi$-exceptional effective $\R$-divisor $F$ such that $\pi^*\{\omega\}-\{F\}$ is a K{\"a}hler class (which shows by the way that $X'$ is indeed K{\"a}hler). 
It follows that $\gamma:=\{\theta\}-\{F\}$ is also a K{\"a}hler class. There remains to set 
$\beta:=(1-\e)\{\theta\}+\e \gamma$ and $E:=D+\e F$ with $\e>0$ small enough so that $(1 - \e) \vol (\a) \leq \vol(\beta)\le \vol(\a)$ also holds (recall that $\vol$ is continuous on big classes).
\end{proof}

As a final remark, if $\a=c_{1}(L)$ is the first Chern class of a big line
bundle $L$, then it follows from a theorem of Fujita that $\vol(\a)$ as defined above coincides with the volume of $L$ (see~\cite{Bou1}), defined by:
$$\vol(L)=\lim_{k\to\infty}\frac{n!}{k^n}\dim H^0(X,L^{\otimes k})$$

\subsection{Currents with full Monge-Amp{\`e}re mass}

As we saw in Proposition~\ref{pro:finite}, if $T_1,...,T_p$ are arbitrary closed positive $(1,1)$-currents on $X$ their non pluripolar product $\langle T_1\wedge...\wedge T_p\rangle$ is always well-defined. 
We want to give here a more convenient description of such global non-pluripolar products.

As a first observation, note that for any K{\"a}hler form $\omega$ on $X$ we have 
$$\langle T_1\wedge...\wedge T_p\rangle=\lim_{\e\to 0}\langle\bigwedge_j(T_j+\e\omega)\rangle,$$
simply by multilinearity of non-pluripolar products. 

We can thus assume that the cohomology classes $\a_j:=\{T_j\}$ are big. Let then $T_{j,\min}\in\a_j$ be a positive current with minimal singularities, and recall that $T_{j,\min}$ has small unbounded locus by Demailly's regularization theorem since $\a_j$ is big. If we pick a smooth representative $\theta_j\in\a_j$ and write
$$T_j=\theta_j+dd^c\f_j$$ 
and 
$$T_{j,\min}=\theta_j+dd^c\f_{j,\min},$$ 
then we have the following useful description of the non-pluripolar product:
\begin{equation}\label{equ:nonpolar}
\langle T_1\wedge...\wedge T_p\rangle=\lim_{k\to\infty}{\bf 1}_{\bigcap_j\{\f_j>\f_{\min}-k\}}\langle\bigwedge_j(\theta_j+dd^c\max(\f_j,\f_{j,\min}-k))\rangle,
\end{equation}
where the right-hand side is in fact non-decreasing, so that convergence holds against any bounded measurable function. This fact is a straightforward consequence of the local character of non-pluripolar products in the plurifine topology, which in fact yields
$${\bf 1}_{\bigcap_j\{\f_j>\f_{\min}-k\}}\langle T_1\wedge...\wedge T_p\rangle={\bf 1}_{\bigcap_j\{\f_j>\f_{\min}-k\}}\langle(\theta_j+dd^c\max(\f_j,\f_{j,\min}-k))\rangle.$$
 
As a consequence, Theorem~\ref{thm:comp} yields

\begin{pro}\label{pro:bound} If $\a_1,...,\a_p\in H^{1,1}(X,\R)$ are psef cohomology classes, then 
$$\{\langle T_1\wedge...\wedge T_p\rangle\}\le\langle\a_1 \cdots \a_p\rangle$$
in $H^{p,p}(X,\R)$ for all positive currents $T_i\in\a_i$.
\end{pro}
\begin{proof} We can assume that each $\a_i$ is big. By duality, we have to show that 
$$\int_X\langle T_1\wedge...\wedge T_p\rangle\wedge\tau\le\langle\a_1 \cdots \a_p\rangle\cdot\{\tau\}$$
for every positive $dd^c$-closed positive form $\tau$. But (\ref{equ:nonpolar}) shows that
$$\int_X\langle T_1\wedge...\wedge T_p\rangle\wedge\tau=\lim_{k\to\infty}\int_{\bigcap_j\{\f_j>\f_{\min}-k\}}\langle\bigwedge_j(\theta_j+dd^c\max(\f_j,\f_{j,\min}-k))\rangle\wedge\tau$$
$$\le\lim_{k\to\infty}\int_X\langle\bigwedge_j(\theta_j+dd^c\max(\f_j,\f_{j,\min}-k))\rangle\wedge\tau,$$
The integrals in question are equal to 
$$\langle\a_1 \cdots \a_p\rangle\cdot\{\tau\}$$
for all $k$, hence the result. 
\end{proof}

When the $\a_i$'s are big, equality holds in Proposition~\ref{pro:bound} when each $T_i\in\a_i$ has minimal singularities. In the general psef case, there might however not exist positive currents achieving equality. Indeed~\cite{DPS} provides an exemple of a nef class $\a$ on a surface such that the only closed positive current in the class is the integration current $T$ on a curve. It follows that $\langle T\rangle=0$, whereas $\langle\a\rangle=\a$ is non-zero. 

\begin{defi}\label{defi:fullMA} 
{\it A closed positive $(1,1)$-current $T$ on $X$ with cohomology class $\a$ will be said to have \emph{full Monge-Amp{\`e}re mass} if
$$\int_X\langle T^n\rangle=\vol(\a).$$}
\end{defi}

We will have a much closer look at such currents in the coming sections. 

\begin{pro} If $T$ is a closed positive $(1,1)$-current on $X$ such that $\langle T^n\rangle\neq 0$, then its cohomology class $\a=\{T\}$ is big.
 \end{pro}
This follows immediately from Theorem 4.8 of~\cite{Bou1}, which was itself an elaboration of the key technical result of~\cite{DP}. 
 
We conjecture the following log-concavity property of total masses:
\begin{conjecture}\label{conj:logconcave} If $T_1,...,T_n$ are closed positive $(1,1)$-currents on $X$, then 
$$\int_X\langle T_1\wedge...\wedge T_n\rangle\ge\left(\int_X\langle T_1^n\rangle\right)^{\frac 1 n}...\left(\int_X\langle T_n^n\rangle\right)^{\frac 1 n}.$$
In particular, the function $T\mapsto(\int_X\langle T^n\rangle)^{\frac 1 n}$ is concave.
\end{conjecture}
Note that the conjecture holds true for currents with \emph{analytic singularities}, since passing to a log-resolution of the singularities reduces it to the Khovanski-Teissier inequalities for nef classes proved in~\cite{Dem3}. Reasoning as in~\cite{Ce2} Theorem 5.5, the general case would follow from the Hodge index-type inequality
$$
\int_X\langle T_1\wedge T_2\wedge\Theta\rangle\ge\left(\int_X\langle T_1^2\wedge\Theta \rangle\right)^{\frac 1 2}\left(\int_X\langle T_2^2\wedge\Theta\rangle\right)^{\frac 1 2}$$
with $\Theta:=T_3\wedge...\wedge T_n$. An important special case of the conjecture will be proved in Corollary~\ref{cor:special} below.

\section{Weighted energy classes}

Let $X$ be a compact K{\"a}hler manifold, and let $T$ be a closed positive $(1,1)$-current on $X$. We write  $T=\theta+dd^c\f$, where $\theta$ is a smooth representative of the cohomology class $\a:=\{T\}$ and $\f$ is a $\theta$-psh function. Recall (Definition~\ref{defi:fullMA}) that $T$ is said to have full Monge-Amp{\`e}re mass if
$$
\int_X\langle T^n\rangle=\vol(\a).
$$
Note that this is always the case when $\a$ is not big, since $\vol(\a)=0$ in that case. Assuming from now on that $\a$ is big, our goal will be to characterize the full Monge-Amp{\`e}re mass property as a finite (weighted) energy condition for $\f$ (Proposition~\ref{pro:carac}). As a consequence we will prove continuity results for non-pluripolar products. 

Since we focus on $\f$ rather than $T$, we introduce

\begin{defi} 
{\it Let $\theta$ be a smooth closed $(1,1)$-form on $X$ and let $\f$ be a $\theta$-psh function. 
\begin{itemize}
\item The \emph{non-pluripolar Monge-Amp{\`e}re measure} of $\f$ (with respect to $\theta$) is
$$
\MA(\f):=\langle(\theta+dd^c\f)^n\rangle.
$$
\item We will say that $\f$ has \emph{full Monge-Amp{\`e}re mass} if $\theta+dd^c\f$ has full Monge-Amp{\`e}re mass, that is iff the measure $\MA(\f)$ satisfies
$$\int_X\MA(\f)=\vol(\a).$$
\end{itemize}}
\end{defi}

Let us stress that $MA(\f)$ is well defined for {\it any}
$\theta$-psh function $\f$.
Note that $\theta$-psh functions with minimal singularities have full Monge-Amp{\`e}re mass. 

\medskip

From now on we fix a smooth closed $(1,1)$-form $\theta$ whose cohomology class $\a\in H^{1,1}(X,\R)$ is big. We also fix the choice of a $\theta$-psh function with minimal singularities $\f_{\min}$. If $\f$ is a given $\theta$-psh function, we set
$$\f^{(k)}:=\max(\f,\f_{\min}-k).$$ 
The $\theta$-psh functions $\f^{(k)}$ have minimal singularities, and they decrease pointwise to $\f$. We will call them the "canonical" approximants of $\f$. As was explained above, for each Borel subset $E$ of $X$ we have
$$
\int_E\MA(\f):= \lim\nearrow \int_{E\cap\{\f>\f_{\min}-k\}}\MA(\f^{(k)}).
$$

\subsection{Comparison principles}

We begin with the following generalized comparison principle, which will be a basic tool in what follows.

\begin{pro}[Generalized comparison principle]\label{pro:comp} Let $T_j=\theta_j+dd^c\f_j$ be closed positive $(1,1)$-currents with cohomology class $\{T_j\}=:\a_j$, $j=0,...,p$. Let also $S_0=\theta_0+dd^c\p_0$ be another positive current in $\a_0$. Then we have
$$\int_{\{\f_0<\p_0\}}\langle S_0^{n-p}\wedge T_1\wedge...\wedge T_p\rangle$$
$$\le\int_{\{\f_0<\p_0\}}\langle T_0^{n-p}\wedge T_1\wedge...\wedge T_p\rangle$$
$$+\langle\a_0^{n-p}\cdot\a_1 \cdots \a_p\rangle-\int_X\langle T_0^{n-p}\wedge T_1\wedge...\wedge T_p\rangle.$$
\end{pro}

\begin{proof} Since non-pluripolar products are local in plurifine topology we have
$$\langle\a_0^{n-p}\cdot\a_1\cdot...\cdot\a_p\rangle\ge\int_X\langle(\theta_0+dd^c\max(\f_0,\p_0-\e))^{n-p}\wedge T_1\wedge...\wedge T_p\rangle$$
$$\ge\int_{\{\f_0<\p_0-\e\}}\langle S_0^{n-p}\wedge T_1\wedge...\wedge T_p\rangle+\int_{\{\f_0>\p_0-\e\}}\langle T_0^{n-p}\wedge T_1\wedge...\wedge T_p\rangle$$
$$\ge\int_{\{\f_0<\p_0-\e\}}\langle S_0^{n-p}\wedge T_1\wedge...\wedge T_p\rangle+\int_X\langle T_0^{n-p}\wedge T_1\wedge...\wedge T_p\rangle-\int_{\{\f_0<\p_0\}}\langle T_0^{n-p}\wedge T_1\wedge...\wedge T_p\rangle$$
and the result follows by letting $\e\to 0$ by monotone convergence. 
\end{proof}

\begin{cor}\label{cor:comp} 
For any two $\theta$-psh functions $\f,\p$ we have
$$\int_{\{\f<\p\}}\MA(\p)\le\int_{\{\f<\p\}}\MA(\f)
+\vol(\a)-\int_X\MA(\f).$$
\end{cor}

The main new feature of this generalized comparison principle is  the additional "error term" 
$\vol(\a)-\int_X\MA(\f)$, which is non-negative by Proposition~\ref{pro:bound}.

\begin{rem}\label{rem:refined} 
{\it If $\f$ and $\psi$ have small unbounded locus (and conjecturally for arbitrary $\theta$-psh functions), the same proof plus Theorem~\ref{thm:comp} yield
$$\int_{\{\f<\psi\}}\MA(\psi)\le\int_{\{\f<\psi\}}\MA(\f)$$
as soon as $\f$ is less singular than $\psi$, that is $\f\ge\psi+O(1)$. Indeed Theorem~\ref{thm:comp} implies 
$$\int_X\MA(\max(\f,\psi-\e))=\int_X\MA(\f)$$ 
in that case, and the result follows from the above proof.}
\end{rem}

We deduce from the comparison principle its usual companion:

\begin{cor}[Domination principle]\label{cor:dom} Let $\p,\f$ be $\theta$-psh functions. If $\f$ has minimal singularities and $\p\le\f$ holds a.e.~wrt $\MA(\f)$, then $\p\le\f$ holds everywhere. 
\end{cor}
\begin{proof} Since $\{\theta\}$ is big by assumption, we can choose a $\theta$-psh function $\rho$ such that $\theta+dd^c\rho\ge\omega$ for some (small enough) K{\"a}hler form $\omega$. It follows that $\MA(\rho)$ dominates Lebesgue measure. We may also assume that $\rho\le\f$ since $\f$ has minimal singularities. 

Now let $\e>0$. Since $\f$ has minimal singularities we have $\vol(\a)=\int_X\MA(\f)$ and the comparison principle thus yields 
$$\e^n\int_{\{\f<(1-\e)\p+\e\rho\}}\MA(\rho)\le\int_{\{\f<(1-\e)\p+\e\rho\}}\MA((1-\e)\p+\e\rho)$$
$$\le\int_{\{\f<(1-\e)\p+\e\rho\}}\MA(\f).$$
But the latter integral is zero by assumption since $\rho \leq \f$ implies 
$$\{\f<(1-\e)\p+\e\rho\}\subset\{\f<\p\}.$$
Since $\MA(\rho)$ dominates Lebesgue measure, we conclude that for each $\e>0$ we have
$$\f\ge(1-\e)\p+\e\rho$$
a.e.~wrt Lebesgue measure, and the result follows. 
\end{proof} 

\begin{rem} 
{\it We do not know whether the result still holds if we replace the assumption that $\f$ has minimal singularities by $\f\ge\p+O(1)$, even if we further assume that $\f$ has finite Monge-Amp{\`e}re energy.}
\end{rem}

\subsection{Weighted energy functionals} 

By a \emph{weight function}, we will mean a smooth increasing function $\chi:\R\to\R$ such that $\chi(-\infty)=-\infty$ and $\chi(t)=t$ for $t\ge 0$.  

\begin{defi} 
{\it Let $\chi$ be a weight function. We define the $\chi$-energy of a $\theta$-psh function with minimal singularities $\f$ as 
\begin{equation}\label{equ:chi_en}
E_\chi(\f):=\frac{1}{n+1}\sum_{j=0}^n\int_X(-\chi)(\f-\f_{\min})\langle T^j\wedge T_{\min}^{n-j}\rangle
\end{equation} 
with $T=\theta+dd^c\f$ and $T_{\min}=\theta+dd^c\f_{\min}$.}
\end{defi} 

Recall that we have fixed an arbitrary $\theta$-psh function $\f_{\min}$ with minimal singularities on $X$.
Note that the $\chi$-energy functional for $\chi(t)=t$ is nothing but the Aubin-Mabuchi energy functional (up to a minus sign, cf.~\cite{Aub,Mab} and~\cite{BB} for the extension to the singular setting). 
 
 \smallskip
 
The next proposition contains important properties of the $\chi$-energy.

\begin{pro}\label{pro:monotone} Let $\chi$ be a convex weight function. Then the following properties hold.
\begin{itemize} 
\item[(i)] For any $\theta$-psh function $\f$ with minimal singularities the sequence 
$$j\mapsto\int_X(-\chi)(\f-\f_{\min})\langle T^j\wedge T_{\min}^{n-j}\rangle$$
as in (\ref{equ:chi_en}) is non-decreasing. In particular we have 
$$
E_{\chi}(\f)\le\int_X(-\chi)(\f-\f_{\min})\MA(\f)\le (n+1)E_\chi(\f)
$$
if $\f\le\fm$. 

\item[(ii)] The $\chi$-energy $\f\mapsto E_\chi(\f)$ is non-increasing.

\item[(iii)] If $\f_k$ decreases to $\f$ with minimal singularities, then $E_\chi(\f_k)\to E_\chi(\f)$.
\end{itemize}
\end{pro} 

\begin{proof} 
Let $u:=\f-\f_{\min}\le 0$. Note that $\chi(u)+\f_{\min}$ is also a $\theta$-psh function with minimal singularities, since 
\begin{equation}\label{equ:hess_conv}
dd^c\chi(u)=\chi'(u)dd^cu+\chi''(u)du\wedge d^cu
\ge-(\theta+dd^c\f_{\min}).
\end{equation}
It follows that the bounded functions $u$ and $\chi(u)$ are both differences of quasi-psh functions which are locally bounded on the complement $X-A$ of a closed complete pluripolar subset $A$. We can thus apply the integration-by-parts formula of Theorem~\ref{thm:intbyparts} to get  
$$\int_X(-\chi)(\f-\f_{\min})\langle T^{j+1}\wedge T_{\min}^{n-j-1}\rangle=\int_{X-A}(-\chi)(u)(T_{\min}+dd^c u)\wedge T^j\wedge T_{\min}^{n-j-1}$$ 
$$=\int_{X-A}(-\chi)(\f-\f_{\min})T^j\wedge T_{\min}^{n-j}+\int_{X-A}\chi'(u)du\wedge d^c u\wedge T^j\wedge T_{\min}^{n-j-1}$$
$$\ge\int_X(-\chi)(\f-\f_{\min})\langle T^j\wedge T_{\min}^{n-j}\rangle$$
for $j=0,...,n-1$, and (i) follows. 

In order to prove (ii) let $v\ge 0$ be a bounded function such that $\f+v$ is still $\theta$-psh. We will show that the derivative at $t=0$ of the function $t\mapsto  e (t) := E_\chi(\f+tv)$ is $\le 0$. Indeed integration by parts yields
\begin{eqnarray*}
- \ e' (0) & = & \sum_{j=0}^n \int_{X-A} v \chi'(u)T^j \wedge T_{\min}^{n-j} + j \chi(u) dd^cv \wedge T^{j-1}\wedge T_{\min}^{n-j} \\
&=& \sum_{j=0}^n\int_{X-A} v\chi'(u)T^j\wedge T_{\min}^{n-j}+jvdd^c\chi(u)\wedge T^{j-1}\wedge T_{\min}^{n-j}\\
& \ge & \sum_{j=0}^n\int_{X-A} v\chi'(u)T^j\wedge T_{\min}^{n-j}+jv\chi'(u)(T-T_{\min})\wedge T^{j-1}\wedge T_{\min}^{n-j} \\
& = & (n+1)\int_{X-A} v\chi'(u)T^n\ge 0,
\end{eqnarray*}
by (\ref{equ:hess_conv}).

Finally let us prove (iii). If $\f_k$ decreases to $\f$, then $\chi(\f_k-\f_{\min})+\f_{\min}$ is a sequence of $\theta$-psh functions decreasing towards $\chi(\f-\f_{\min})+\f_{\min}$, thus Bedford-Taylor's theorem implies that 
$$
\lim_k T_k^j\wedge T_{\min}^{n-j} = T^j\wedge T_{\min}^{n-j}
$$
and 
$$
\lim_k \chi(\f_k-\f_{\min})T_k^j\wedge T_{\min}^{n-j}=\chi(\f-\f_{\min})T^j\wedge T_{\min}^{n-j}
$$
weakly on $X-A$. But Theorem~\ref{thm:comp} implies that the total masses also converge, since in fact
$$
\int_{X-A}T_k^j\wedge T_{\min}^{n-j}=\int_{X-A} T^j\wedge T_{\min}^{n-j}
$$ 
for all $k$, the thus trivial extensions converge on the whole of $X$, 
i.e.
$$
\lim_k\langle T_k^j\wedge T_{\min}^{n-j}\rangle=\langle T^j\wedge T_{\min}^{n-j}\rangle
$$
Since $\chi(\f_k-\f_{\min})$ is uniformly bounded, it follows from Bedford-Taylor's monotone convergence theorem~\cite{BT2} that 
$$
\lim_k\int_X\chi(\f_k-\f_{\min})\langle T_k^j\wedge T_{\min}^{n-j}\rangle=\int_X \chi(\f-\f_{\min})\langle T^j\wedge T_{\min}^{n-j}\rangle
$$
for all $j$, which concludes the proof.
\end{proof}

\begin{defi}  
{\it Let $\chi$ be a convex weight function. 
 If $\f$ is an arbitrary $\theta$-psh function, its $\chi$-energy is defined as 
 $$
 E_\chi(\f):=\sup_{\psi\ge\f} E_\chi(\psi)\in]-\infty,+\infty]
 $$
 over all $\psi\ge\f$ with minimal singularities. We say that $\f$ has finite $\chi$-energy if $E_{\chi} (\f) < + \infty$.}
\end{defi} 

We will see in Corollary \ref{cor:formule_en} how to extend  (\ref{equ:chi_en})  to this
more general picture.

\begin{pro}\label{pro:semicont} 
Let $\chi$ be a convex weight. 
\begin{itemize}
\item[(i)] The $\chi$-energy functional $\f\mapsto E_\chi(\f)$ is non-increasing and lower semi-continuous on $\theta$-psh functions for the $L^1(X)$ topology. 
\item[(ii)] $\f\mapsto E_\chi(\f)$ is continuous along decreasing sequences.
\item[(iii)] If $\f_j$ are $\theta$-psh functions converging to $\f$ in the $L^1(X)$-topology, then 
$$
\sup_j E_\chi(\f_j)<+\infty
$$
implies that $\f$ has finite $\chi$-energy.
\end{itemize} 
\end{pro}

\begin{proof} The last two points are trivial consequences of the first one. Let $\f_j\to\f$ be a convergent sequence of $\theta$-psh functions, and set $\widetilde{\f}_j:=(\sup_{k\ge j}\f_k)^*$, so that $\tilde{\f_j}\ge\f_j$ decreases pointwise to $\f$. If $\psi\ge\f$ is a given $\theta$-psh function with minimal singularities, then $\max(\widetilde{\f}_j,\psi)\ge\max(\f_j,\psi)$ decreases to $\max(\f,\psi)=\psi$, hence 
$$
E_\chi(\psi)=\lim_j E_\chi(\max(\widetilde{\f}_j,\psi))\le\liminf_j E_\chi(\f_j)
$$
by Proposition~\ref{pro:monotone}, and the result follows by definition of $E_\chi(\f)$. 
\end{proof} 

As an important consequence, we get 
$$
E_\chi(\f)=\sup_k E_\chi(\f^{(k)})
$$
for the "canonical" approximants $\f^{(k)} := \max(\f,\f_{\min}-k)$ of $\f$.
The link between weighted energies and full Monge-Amp\`ere mass is given by the
following result:

\begin{pro}\label{pro:carac}  Let $\f$ be a $\theta$-psh function.
\begin{itemize} 
\item[(i)] The function $\f$ has full Monge-Amp{\`e}re mass iff $E_\chi(\f)<+\infty$ for some convex weight $\chi$
\item[(ii)] If $\f$ has full Monge-Amp{\`e}re mass and $\chi$ is a given convex weight, then 
$$E_\chi(\f)<+\infty\,\text{ iff }\int_X(-\chi)(\f-\f_{\min})\MA(\f)<+\infty.$$
\end{itemize}  
\end{pro} 
\begin{proof} Upon subtracting a constant we may assume that $\f\le\fm$. Let $\f^{(k)}:=\max(\f,\f_{\min}-k)$. In view of Proposition~\ref{pro:monotone}, we have to show that $m_k:=\int_{\{\f\le\f_{\min}-k\}}\MA(\f^{(k)})$ tends to $0$ 
iff 
\begin{equation}\label{equ:intbornee}
\sup_k\int_X(-\chi)(\f^{(k)}-\f_{\min})\MA(\f^{(k)})<+\infty
\end{equation}
for some convex weight $\chi$. But we have
\begin{eqnarray*}
\lefteqn{\int_X(-\chi)(\f^{(k)}-\f_{\min})\MA(\f^{(k)}) \hskip 3 cm \text{ } } \\
& = & \vert\chi(-k)\vert m_k + \int_{\{\f>\f_{\min}-k\}}(-\chi)(\f-\f_{\min})\MA(\f) \\
& = & \vert\chi(-k)\vert m_k + O (1),
\end{eqnarray*}
for some convex weights $\chi$. Indeed we may always choose $\chi$ such that 
$$
\int_X(-\chi)(\f-\f_{\min})\MA(\f)<+\infty,
$$ 
simply because $\MA(\f)$ puts no mass on the pluripolar set $\{\f=-\infty\}$. It follows that (\ref{equ:intbornee}) holds for some convex weight iff there exists a convex weight $\chi$ such that $\chi(-k) =O(m_k^{-1})$, which is indeed the case iff $m_k\to 0$.

Let's now prove the second point. Assume that $E_\chi(\f)<+\infty$. Since $\MA(\f)$ is the increasing limit in the strong Borel topology of 
$${\bf 1}_{\{\f>\f_{\min}-k\}}\MA(\f^{(k)})$$ 
and since $\chi(\f^{(l)}-\f_{\min})$ is a bounded measurable function, we infer that

\begin{eqnarray*}
\int_X(-\chi)(\f^{(l)}-\f_{\min})\MA(\f)& = & \lim_k\int_{\{\f>\f_{\min}-k\}}(-\chi)(\f^{(l)}-\f_{\min})\MA(\f^{(k)}) \\ 
& \le & \lim_k\int_{\{\f>\f_{\min}-k\}}(-\chi)(\f^{(k)} - \f_{\min})\MA(\f^{(k)}) \\
& \le & E_\chi(\f),
\end{eqnarray*}
since $\f^{(k)}\le\f^{(l)}$ for $k \ge l$, and we infer that $\int_X(-\chi)(\f-\f_{\min})\MA(\f)<+\infty$ 
by monotone convergence as desired.  

Conversely, if $(-\chi)(\f-\f_{\min})\in L^1(\MA(\f))$ \emph{and} $\f$ has full Monge-Amp{\`e}re mass, then Lemma~\ref{lem:comp} below implies that $E_\chi(\f)=\sup_k E_\chi(\f^{(k)})<+\infty$.
\end{proof}

\begin{lem}\label{lem:comp} Let $\f$ be a $\theta$-psh function with canonical approximants $\f^{(k)}$. Then for every weight function $\chi$ we have:
$$\int_X(-\chi)(\f^{(k)}-\f_{\min})\MA(\f^{(k)})\le\int_X(-\chi)(\f-\f_{\min})\MA(\f)$$
$$+(-\chi)(-k)\left(\vol(\a)-\int_X\MA(\f)\right)$$
\end{lem}
\begin{proof} We have 
$$\int_X(-\chi)(\f^{(k)}-\f_{\min})\MA(\f^{(k)})$$
$$=(-\chi)(-k)\int_{\{\f\le\f^{(k)}\}}\MA(\f^{(k)})+\int_{\{\f>\f_{\min}-k\}}\chi(\f-\f_{\min})\MA(\f)$$
(since ${\bf 1}_{\{\f>\f_{\min}-k\}}\MA(\f^{(k)})={\bf 1}_{\{\f>\f_{\min}-k\}}\MA(\f)$ by plurifine locality)
$$\le(-\chi)(-k)\int_{\{\f\le\f^{(k)}\}}\MA(\f)+(-\chi)(-k)\left(\vol(\a)-\int_X\MA(\f)\right)$$
$$+\int_{\{\f>\f_{\min}-k\}}\chi(\f-\f_{\min})\MA(\f)$$
by the generalized comparison principle, and the result follows.
\end{proof}

\subsection{Energy classes of positive currents}

Since $\chi$ is convex with $0\le\chi'\le 1$, we have 
$$\chi(t)\le\chi(t+C)\le\chi(t)+C$$
for all $C>0$, and it easily follows that the condition $\f$ has finite $\chi$-energy doesn't depend on the choice of the function with minimal singularities $\f_{\min}$ with respect to which $E_\chi$ is defined. 

\begin{defi} 
{\it Given a convex weight $\chi$ we denote by $\cE_\chi(\a)$ the set of positive currents $T\in\a$ with finite $\chi$-energy, and by $\cE(\a)$ the set of positive currents with full Monge-Amp{\`e}re mass. }
\end{defi} 

Proposition~\ref{pro:carac} can thus be reformulated by saying that 
$$\cE(\a)=\bigcup_\chi\cE_\chi(\a)$$
where $\chi$ ranges over all convex weights.

\begin{pro}\label{pro:convex} The energy classes satisfy the following properties:
\begin{itemize} 
\item If $T\in\cE_\chi(\a)$, then also $S\in\cE_\chi(\a)$ for all positive currents $S\in\a$ less singular than $T$.
\item The set $\cE(\a)$ is an extremal face (in particular a convex subset) of the closed convex set of all positive currents in $\a$. 
\end{itemize}
\end{pro}
Recall that an \emph{extremal face} $F$ of a convex set $C$ is a convex subset of $C$ such that for all $x,y\in C$, $tx+(1-t)y\in F$ for some $0<t<1$ implies $x,y\in F$. 

\begin{proof} The first point follows from Proposition~\ref{pro:carac} above, since $E_\chi$ is non-increasing. To prove the second point, we first take care of the extremality property. Let $T, T'\in\a$ be two closed positive currents and $0<t<1$ be such that $tT+(1-t)T'$ has finite energy. We have to show that $T$ also has finite energy (then exchange the roles of $T$ and $T'$).
If $T_{\min}\in\a$ is a positive current with minimal singularities, then $tT+(1-t)T_{\min}$ is less singular than $tT+(1-t)T'$, thus it \emph{a fortiori} has finite energy, i.e.~we can assume that $T'=T_{\min}$ has minimal singularities. Now write $T=\theta+dd^c\f$, $T_{\min}=\theta+dd^c\f_{\min}$, and set $\f_t:=t\f+(1-t)\f_{\min}$. The assumption is that 
$$\lim_k\int_{\{\f_t\le\f_{\min}-k\}}\MA(\max(\f_t,\f_{\min}-k))=0,$$
and we are to show that the same holds with $\f$ in place of $\f_t$. But note that
$$\max(\f_t,\f_{\min}-k)=\max(t(\f-\f_{\min},-k)+\f_{\min}$$
$$=t\max(\f,\f_{\min}-kt^{-1})+(1-t)\f_{\min}$$
hence
$$\int_{\{\f_t\le\f_{\min}-k\}}\MA(\max(\f_t,\f_{\min}-k))$$
$$\ge t^n\int_{\{\f\le\f_{\min}-kt^{-1}\}}\MA(\max(\f,\f_{\min}-kt^{-1}))$$
and the result follows. 

Convexity of $\cE(\a)$ now follows exactly as in the proof of Proposition 1.6 in~\cite{GZ2}. 
\end{proof} 
Convexity of $\cE(\a)$ would also follow from Conjecture~\ref{conj:logconcave} above. Conversely the following important special case of the conjecture can be settled using convexity.

\begin{cor}\label{cor:special} Let $\a$ be a big class, and let $T_1,...,T_n\in\a$ be positive currents with full Monge-Amp{\`e}re mass. Then 
$$\int_X\langle T_1\wedge...\wedge T_n\rangle=\langle\a^n\rangle.$$
\end{cor} 
\begin{proof} Upon scaling $\a$, we can assume that $\langle\a^n\rangle=1$. For any $t=(t_1,...,t_n)\in\R_+^n$ set 
$$P(t_1,...,t_n):=\int_X\langle\left(t_1T_1+...+t_nT_n\right)^n\rangle,$$
which defines a homogeneous polynomial of degree $n$ by multilinearity of non-pluripolar products.
Since $T_i\in\cE(\a)$ for each $i$, convexity implies that $t_1T_1+...+t_nT_n\in\cE(\a)$ when the $t_i$'s lie on the simplex 
$$\Delta:=\{t\in\R_+^n,t_1+...+t_n=1\}.$$
We thus see that the polynomial
$$P(t_1,...,t_n)-(t_1+...+t_n)^n$$ 
vanishes identically on $\Delta$. Since this polynomial is homogeneous of degree $n$, it therefore vanishes on $\R_+^n$, and this implies that 
$$P(t_1,...,t_n)=(t_1+...+t_n)^n$$
for all $t$. Identifying the coefficients of the monomial $t_1...t_n$ on both sides now yields the result. 
\end{proof}

\begin{rem} 
{\it It is quite plausible that each $\cE_\chi(\a)$ is itself convex, and indeed for $\chi(t)=t$ it simply follows from the convexity of the Aubin-Mabuchi energy functional $E_1$ (cf.~\cite{BB} for a proof in this singular situation). The functional $E_\chi$ is however \emph{not} convex when $\chi$ is not affine, since $E_\chi(\f+c)$ is rather a concave function of $c\in\R$.}
\end{rem}

\subsection{Continuity properties of Monge-Amp{\`e}re operators}

The Monge-Amp{\`e}re operator on functions of full Monge-Amp{\`e}re mass satisfies the usual continuity properties along monotonic sequences: 

\begin{thm}\label{thm:cont} If $\f$ is a $\theta$-psh function with full Monge-Amp{\`e}re mass and $\f_j$ is a sequence of $\theta$-psh functions with full Monge-Amp{\`e}re mass decreasing (resp.~increasing a.e.) to $\f$, then
$$\MA(\f_j)\to\MA(\f)$$
as $j\to\infty$. If $\f$ furthermore has finite $\chi$-energy for a given convex weight $\chi$, then 
$$\chi(\f_j-\f_{\min})\MA(\f_j)\to\chi(\f-\f_{\min})\MA(\f)$$
weakly as $j\to\infty$.
\end{thm} 
\begin{proof} Theorem~\ref{thm:cont} has already been proved for functions with minimal singularities during the proof of Proposition~\ref{pro:monotone}. Let us first prove the first point in the general case. Choose $\chi$ such that $\f$ and $\f_1$ both have finite $\chi$-energy, so that $E_\chi(\f_j)\le C$ uniformly by monotonicity of $E_\chi$. Let $h$ be a continuous function on $X$. For each $k$ we have 
$$\lim_j\int_X h\,\MA(\f_j^{(k)})=\int_X h\,\MA(\f^{(k)})$$
by the minimal singularities case, thus it is enough to establish that 
$$\lim_k\int_X h\,\MA(\psi^{(k)})=\int_X h\,\MA(\psi)$$
uniformly for all $\psi\ge\f$. We have
$$\left|\int_X h\left(\MA(\psi^{(k)})-\MA(\psi)\right)\right|$$
$$\le\sup_X|h|\left(\int_{\{\psi\le\f_{\min}-k\}}\MA(\psi^{(k)})+\int_{\{\psi\le\f_{\min}-k\}}\MA(\psi)\right)$$
since $\MA(\psi^{(k)})=\MA(\psi)$ on $\{\psi>\f_{\min}-k\}$. The first term is controlled by 
$$|\chi|(-k)^{-1}\int_X(-\chi)(\psi^{(k)})\MA(\psi^{(k)})\le|\chi|(-k)^{-1}E_\chi(\f),$$ 
hence tends to $0$ uniformly with respect to $\psi\ge\f$. On the other hand, since $\MA(\psi)$ is the monotone and thus \emph{strong} limit of ${\bf 1}_{\{\psi>\f_{\min}-l\}}\MA(\psi^{(l)})$, the second term writes
$$\int_{\{\psi\le\f_{\min}-k\}}\MA(\psi)=\lim_l\int_{\{\f_{\min}-l<\psi\le\f_{\min}-k\}}\MA(\psi^{(l)})$$
which is similarly controlled by 
$$\limsup_l|\chi|(-k)^{-1}E_\chi(\psi^{(l)})\le|\chi|(-k)^{-1}E_\chi(\f),$$ 
hence the result. The same reasoning shows that 
$$\chi(\f_j-\f_{\min})\MA(\f_j)\to\chi(\f-\f_{\min})\MA(\f)$$
as soon as $\f$ has finite $\tilde{\chi}$-energy for some convex weight $\tilde{\chi}\gg\chi$. But this is in fact automatic, since $E_\chi(\f)<+\infty$ is equivalent to $(-\chi)(\f-\f_{\min})\in L^1(\MA(\f))$ by Proposition~\ref{pro:carac} since $\f$ has full Monge-Amp{\`e}re mass, and this integrability condition can always be improved to $\widetilde{\chi}(\f-\f_{\min})\in L^1(\MA(\f))$ for $\widetilde{\chi}\gg\chi$ by a standard measure theoretic result (see ~\cite{GZ2} for more details on such properties).
\end{proof} 

As a consequence, we get a more explicit formula for the $\chi$-energy. 

\begin{cor}\label{cor:formule_en} If $\f$ has full Monge-Amp{\`e}re mass and $\chi$ is a given convex weight, then setting $T=\theta+dd^c\f$ and $T_{\min}=\theta+dd^c\f_{\min}$ we have
$$E_\chi(\f)=\int_X(-\chi)(\f-\f_{\min})\sum_{j=0}^n\langle T^j\wedge T_{\min}^{n-j}\rangle$$ 
in the sense that the left-hand side is finite iff the right-hand side is, and equality holds in that case.
\end{cor} 

\subsection{A useful asymptotic criterion}

In the continuity theorem above, it is crucial that the functions dealt with have finite energy in order to ensure that no mass is lost towards the polar set in the limit process. For an arbitrary $\theta$-psh function the condition $\chi(\f-\f_{\min})\in L^1(\MA(\f))$ alone doesn't imply that $\f$ has finite energy. For instance, if $\omega$ denotes the Fubini-Study form on $\PP^1$, a global potential of the current $T:=\delta_0$ yields a counter-example since $\langle T\rangle=0$! This example is in fact a simple instance of the type of currents we'll have to handle in the proof of Theorem~\ref{thm:MA}. 

\smallskip

 It is a consequence of Proposition~\ref{pro:semicont} that any $L^1(X)$-limit $\f=\lim_j\f_j$ of functions $\f_j$ with \emph{full Monge-Amp{\`e}re mass} and such that 
 $$\sup_j\int_X(-\chi)(\f_j-\f_{\min})\MA(\f_j)<+\infty$$ 
has finite $\chi$-energy. This still holds for sequences with asymptotically full Monge-Amp{\`e}re mass as shown by our next result which yields a practical criterion to check the finite energy condition.

\begin{pro}\label{pro:presquefini} Let $\chi$ be a convex weight and $\f_j\to\f$ be a sequence of $\theta$-psh functions converging in $L^1(X)$ to a 
$\theta$-psh function $\f$. Assume that 
\begin{enumerate}
\item $\int_X\MA(\f_j)\to\vol(\a)$.
\item $\sup_j\int_X(-\chi)(\f_j-\f_{\min})\MA(\f_j)<+\infty$.
\end{enumerate}
Then $\f$ has finite $\chi$-energy. 
\end{pro}

\begin{proof} Lemma~\ref{lem:comp} above yields two constants $A,B>0$ such that
$$E_{\chi}(\f_j^{(k)})\le A+B(-\chi)(-k)\left(\vol(\a)-\int_X\MA(\f_j)\right)$$
hence 
$$E_\chi(\f^{(k)})\le\liminf_j E_\chi(\f_j^{(k)})\le A$$
for all $k$ by Proposition~\ref{pro:semicont}, and the result follows.   
\end{proof}

\begin{pro}\label{pro:monge_max} Let $\f_j$ be a sequence of arbitrary $\theta$-psh functions uniformly bounded above, and let $\f:=(\sup_j\f_j)^*$. Suppose that $\mu$ is a positive measure such that $\MA(\f_j)\ge\mu$ for all $j$. If $\f$ has full Monge-Amp{\`e}re mass, then $\MA(\f)\ge\mu$. 
\end{pro}

\begin{proof} Set $\psi_j:=\max_{1\le i\le j}\f_i$, so that $\psi_j$ increases a.e.~to $\f$. It is a standard matter to show that $\MA(\psi_j)\ge\mu$ using plurifine locality (cf.~\cite{GZ2}, Corollary 1.10). The desired result does however not follow directly from Theorem~\ref{thm:cont} since $\psi_j$ doesn't \emph{a priori} have full Monge-Amp{\`e}re mass . But since $\MA(\psi_j)$ and $\MA(\psi_j^{(k)})$ coincide on $\{\psi_j>\f_{\min}-k\}$ as Borel measures, we get
\begin{eqnarray*}
\MA(\p_j^{(k)}) & \ge & {\bf 1}_{\{\p_j>\f_{\min}-k\}}\MA(\p_j) \\
& \ge & {\bf 1}_{\{\p_j>\f_{\min}-k\}}\mu \\
& \ge & {\bf 1}_{\{\p_1>\f_{\min}-k\}}\mu.
\end{eqnarray*}

For each $k$ fixed Theorem~\ref{thm:cont} (in fact simply Beford-Taylor on a Zariski open subset plus constancy of total masses) yields 
$$\lim_j\MA(\psi_j^{(k)})=\MA(\f^{(k)}),$$
hence 
$$\MA(\f^{(k)})\ge{\bf 1}_{\{\p_1>\f_{\min}-k\}}\mu.$$
Now the left hand side converges to $\MA(\f)$ weakly as $k\to\infty$ since $\f$ is assumed to have finite Monge-Amp{\`e}re energy, whereas the right hand side tends to $\mu$ since the latter puts no mass on the pluripolar set $\{\p_1=-\infty\}$, being dominated by non-pluripolar measures $\MA(\f_j)$.  
\end{proof} 

\begin{cor}\label{cor:monge_max} Let $T_j\to T$ be a convergent sequence of positive currents in a given cohomology class $\a$, and assume that the limit current $T$ has \emph{full Monge-Amp{\`e}re mass}. Assume that a positive measure $\mu$ and non-negative functions $f_j$ are given such that 
$$\langle T_j^n\rangle\ge f_j\mu$$
holds for all $j$. Then we have
$$\langle T^n\rangle\ge(\liminf_{j\to\infty}f_j)\mu.$$
\end{cor}
\begin{proof} Write $T_j=\theta+dd^c\f_j$ with $\f_j$ normalized in some fixed way and set 
$$\widetilde{\f}_j:=(\sup_{k\ge j}\f_k)^*$$ 
so that $\widetilde{\f}_j$ decreases pointwise to $\f$ with $T=\theta+dd^c\f$. Since $\f$ has full Monge-Amp{\`e}re mass, so does each $\widetilde{\f}_j\ge\f$, hence $\MA(\widetilde{\f}_j)\ge(\inf_{k\ge j}f_k)\mu$ by Proposition~\ref{pro:monge_max}. The result now follows by Theorem~\ref{thm:cont}. 
\end{proof}

\section{Non-pluripolar measures are Monge-Amp{\`e}re}

We will say for convenience that a positive measure $\mu$ is \emph{non-pluripolar} if it puts no mass on pluripolar subsets. The range of the non-pluripolar Monge-Amp{\`e}re operator is described by the following result:

\begin{thm}\label{thm:MA} Let $X$ be a compact K{\"a}hler manifold and let $\a\in H^{1,1}(X,\R)$ be a big cohomology class. Then for any non-pluripolar measure $\mu$ such that $\mu(X)=\vol(\a)$ there exists a unique closed positive current $T\in\a$ such that
$$\langle T^n\rangle=\mu.$$ 
\end{thm}
Note that $T$ automatically has full Monge-Amp{\`e}re mass since
$$\int_X\langle T^n\rangle=\mu(X)=\vol(\a)$$
by assumption. 

In the sequel $\theta$ denotes a smooth representative of $\a$. We then denote as before by $\MA(\f):=\langle(\theta+dd^c\f)^n\rangle$ the non-pluripolar Monge-Amp{\`e}re measure of a $\theta$-psh function $\f$. We also choose once and for all a $\theta$-psh function $\f_{\min}$ with minimal singularities. 

\subsection{Outline of the existence proof} 
We first present an outline of the existence proof, and provide technical details below. By Proposition~\ref{cor:approxzar}, given $\e>0$ there exists a decomposition 
$$\pi^*\a=\beta+\{E\}$$
on a modification $\pi:X'\to X$ with $\beta$ a K{\"a}hler class, $E$ an effective $\R$-divisor on $X'$ and 
$$
\vol(\a)-\e\le\vol(\beta)\le\vol(\a)
$$ 
(the whole data depends on $\e$). By the main result of~\cite{GZ2} (i.e.~the existence part of Theorem~\ref{thm:MA} for a K{\"a}hler class), there exists a positive current with finite energy $S\in\beta$ such that $$\langle S^n\rangle=c\mu'$$ 
holds on $X'$ with 
$$c:=\vol(\beta)/\vol(\a)\le 1$$ 
and $\mu'$ denoting the non-pluripolar measure on $X'$ induced by lifting $\mu$ on the Zariski open subset where $\pi$ is an isomorphism. We will first show that the pushed-forward current  
$$\pi_*(S+[E])=:\theta+dd^c\f$$ 
satisfies $\MA(\f)=\langle(\theta+dd^c\f)^n\rangle=c\mu$, which is at least plausible since $[E]$ vanishes outside a pluripolar subset. We would already be done at that stage were it the case that $c=1$, i.e.~$\vol(\beta)=\vol(\a)$, but this actually \emph{never} happens unless $\a$ was a K{\"a}hler class in the first place. On the other hand we have obtained at that point a sequence of $\theta$-psh functions $\f_j$ such that 
$$\MA(\f_j)=(1-\e_j)\mu$$ 
with $\e_j>0$ decreasing towards $0$. We will prove next that some $L^1(X)$-limit $\f_\infty$ of the $\f_j$'s (appropriately normalized) solves $\MA(\f_\infty)=\mu$, and setting $T:=\theta+dd^c\f_\infty$ then proves the theorem.  

It is in fact enough to show that $\f_\infty$ has full Monge-Amp{\`e}re mass. Indeed Corollary~\ref{cor:monge_max} then implies that $\MA(\f_\infty)\ge\mu$, hence equality since both measures have the same mass, and we will be done. Since 
$$\lim_{j\to\infty}\int_X\MA(\f_j)=\vol(\a),$$ 
Proposition~\ref{pro:presquefini} reduces us to finding a convex weight $\chi$ such that
$$\sup_j\int_X(-\chi)(\f_j-\f_{\min})\MA(\f_j)<+\infty$$
But we will prove in Proposition~\ref{pro:nonpluri} below that non-pluripolarity of $\mu$ yields a convex weight $\chi$ such that $\int_X(-\chi)(\psi)d\mu$ is uniformly bounded for all normalized $\theta$-psh functions $\psi$, hence the result since $\MA(\f_j)\le\mu$ for all $j$.

\subsection{Existence proof: technical details}
Using the above notations, let $\omega\in\beta$ be a K{\"a}hler form and write $S=\omega+dd^c\psi$. We first have to show that 
$$\langle\left(\pi_*(S+[E])\right)^n\rangle=\pi_*\langle S^n\rangle,$$
that is
\begin{equation}\label{equ:mes}
\langle(\theta+dd^c\f)^n\rangle=\pi_*\langle(\omega+dd^c\psi)^n\rangle.
\end{equation}
At that stage everything depends on $\e$, which is however fixed for the moment. 
 Let $\Omega$ be a Zariski open subset of $X$ such that $\f_{\min}$ is locally bounded on $\Omega$, $\pi$ induces an isomorphism $\Omega':=\pi^{-1}(\Omega)\to\Omega$ and every component of $E$ is contained in the complement of $\Omega'$. Since the measures involved in equation~(\ref{equ:mes}) put no mass on $X-\Omega$, it is enough to show the result on $\Omega$. Let thus $f$ be a given smooth function with compact support in $\Omega$. We then have by definition 
  
 $$\int_X f\langle(\theta+dd^c\f)^n\rangle=\lim_k\int_{\{\f>\f_{\min}-k\}\cap\Omega}f(\theta+dd^c\max(\f,\f_{\min}-k))^n$$
 $$=\lim_k\int_{\{\f'>\f_{\min}'-k\}\cap\Omega'}f'(\theta'+dd^c\max(\f',\f_{\min}'-k))^n$$
where primed objects have been pulled back by $\pi$. Note that $\f_{\min}'$ is a $\theta'$-psh function with minimal singularities by Proposition~\ref{pro:tire_min}. 

Now let $\theta_E:=\theta'-\omega$, so that 
$$[E]=\theta_E+dd^c\f_E$$ 
for some $\theta_E$-psh function $\f_E$. For cohomological reasons we have 
$$\theta'+dd^c\f'=\pi^*\pi_*(\omega+dd^c\psi+[E])=\omega+dd^c\psi+[E]=\theta'+dd^c(\psi+\f_E)$$
since the cohomology class of the right-hand current is $\pi^*\theta$. 
We thus see that $\f'=\psi+\f_E$ up to an additive constant, which can be chosen to be $0$, and we are reduced to showing that
$$I_k:=\lim_k\int_{\{\psi+\f_E>\f_{\min}'-k\}\cap\Omega'}f'(\omega+\theta_E+dd^c\max(\psi+\f_E,\f_{\min}'-k))^n$$
coincides with 
$$\lim_k\int_{\{\psi>-k\}\cap\Omega'}f'(\omega+dd^c\max(\psi,-k))^n.$$
Now $\f_E$ is smooth on $\Omega'$ since $E$ doesn't meet the latter set by assumption, and $\f_{\min}'-\f_E\ge-C$ on $X'$ since the $\theta'$-psh function $\f_{\min}'$ has minimal singularities, hence is less singular than the $\theta'\ge\theta_E$-psh function $\f_E$. We thus see that $\f_{\min}'-\f_E$ is bounded near the support of $f'$. On the other hand, $\theta_E+dd^c\f_E=[E]$ vanishes on $\Omega'$ so that 
$$(\omega+\theta_E+dd^c\max(\psi+\f_E,\f_{\min}'-k))^n=(\omega+dd^c\max(\psi,\f'_{\min}-\f_E-k))^n$$
there, hence
$$I_k=\int_{\{\psi>v-k\}\cap\Omega'}f'(\omega+dd^c\max(\psi,v-k))^n$$
with $v:=\f_{\min}'-\f_E$ bounded near the support of $f'$. We thus conclude that the latter integral equals
$$\lim_k\int_{\{\psi>-k\}\cap\Omega'}f'(\omega+dd^c\max(\psi,-k))^n$$
as desired, which finishes the proof of~(\ref{equ:mes}). 
 
 At this point, we have thus obtained a sequence of $\theta$-psh functions $\f_j$ such that $\MA(\f_j)=(1-\e_j)\mu$ for some sequence $\e_j>0$ decreasing towards $0$. If we normalize them in some way, say by $\sup_X\f_j=0$, we can also assume by compactness that $\f_j$ converges in $L^1(X)$ to some $\theta$-psh function $\f_\infty$ as $j\to\infty$. 
 
As explained above, proving that $\f_\infty$ has full Monge-Amp{\`e}re mass is enough to conclude $\MA(\f_\infty)\ge\mu$ by Corollary~\ref{cor:monge_max}, hence $\MA(\f_\infty)=\mu$ since both measures then have the same mass. 
 
Since $\int_X\MA(\f_j)\to\vol(\a)$, we are reduced by Proposition~\ref{pro:presquefini} to finding a convex weight $\chi$ such that 
$$\sup_j\int_X(-\chi)(\f_j-\f_{\min})\MA(\f_j)\le\sup_j\int_X(-\chi)(\f_j-\f_{\min})d\mu<+\infty,$$
which is taken care of by the next proposition. This concludes the proof of Theorem~\ref{thm:MA}. 

\begin{pro}\label{pro:nonpluri} Let $\mu$ be a non-pluripolar measure and let $\cK$ be a given compact set of quasi-psh functions on $X$. Then there exists a convex weight function $\chi$ and $C>0$ such that 
$$\int_X(-\chi)(\f)d\mu\le C$$
for all $\f\in\cK$. 
\end{pro}
\begin{proof} Functions in $\cK$ are uniformly bounded from above by compactness, and we can thus assume that they are all non-positive. Now let $\omega$ be a given K{\"a}hler form, and let $\cM$ be the (weakly) closed convex set of positive measures generated by all Monge-Amp{\`e}re measures $(\omega+dd^c\psi)^n$ with $0\le\psi\le 1$ is $\omega$-psh. By the Chern-Levine-Nirenberg inequality there exists a constant $C>0$ such that 
$$\int_X(-\f)(\omega+dd^c\psi)^n\le C$$ 
for all $\f\in\cK$ and all $0\le\psi\le 1$ as above, and this remains true on $\cM$ by convexity and lower semi-continuity of $\nu\mapsto\int_X(-\f)d\nu$ (since $-\f$ is lower semi-continuous - just write it as the supremum of all continuous functions below it). 

Now recall that pluripolar subsets $A$ are characterized by the condition that $\nu(A)=0$ for all Monge-Amp{\`e}re measures $\nu$ as above. By an extension of Radon-Nikodym's theorem proved in~\cite{R}, there exist measures $\nu\in\cM$ and $\nu'\perp\cM$ such that our given measure $\mu$ writes
$$\mu=g\nu+\nu'$$
with $g$ a non-negative element of $L^1(\nu)$. Since $\mu(A)=0$ for all pluripolar subsets by assumption, we infer that $\nu'=0$, and we have thus proved that $\mu\ll\nu$ for some $\nu\in\cM$. We now conclude the proof by the following result, certainly a standard one in the theory of Orlicz spaces. 
\end{proof}

\begin{lem} Let $\mu\ll\nu$ be two positive measures with finite mass on a measured space $X$. Then there exists a concave function $h : \mathbb R^+ \longrightarrow \mathbb R^+$ and $C>0$ such that
$$\int_X h (f)d\mu\le C + \int_X f d\nu$$
for all measurable functions $f\ge 0$.
\end{lem}
\begin{proof} Let $g\ge 0$ be the function in $L^1(\nu)$ such that $d\mu=g d\nu$. By an easy measure-theoretic result that we've already used, there exists a convex increasing function $\tau$ on $\R^+$ such that $\tau(x)\gg x$ for $x\to+\infty$ but still $\int_X \tau(g)d \nu<+\infty$. Now let $$\tau^*(y)=\sup_{x\ge 0}\left(xy-\tau(x)\right)$$ 
be the Legendre transform of $\tau$ and let $h$ be the reciprocal function of $\tau^*$. By definition, we thus have the following generalized form of Young's inequality 
$$h (y) x\le \tau(x)+y.$$ 
If follows that 
$$\int_X h (f)d\mu=\int_X h (f)g d \nu \le \int_X \tau(g)d\nu+\int_X f d\nu$$
which yields the desired result with $C:=\int_X \tau(g)d\nu$. 
\end{proof}

\begin{rem} 
{\it When $\cK$ consists of a single function, Proposition~\ref{pro:nonpluri} follows from standard measure-theoretic considerations since $\mu\{\f=-\infty\}=0$. The general case would thus follow by compactness provided that $\f\mapsto\int_X \f d\mu$ were known to be lower semi-continuous on $\theta$-psh functions in $L^1(X)$ topology (it is always \emph{upper} semi-continuous by Fatou's lemma). But already when $n=1$ the lsc condition is not automatic: it holds iff $\mu$ has \emph{continuous} local potentials.}
\end{rem}

\subsection{Uniqueness}
In this section we show how to adapt Dinew's tricky proof of uniqueness in~\cite{Din2} to our more general setting. As we shall see, his arguments carry over \emph{mutatis mutandis}. On the other hand, we will make one of his arguments more precise, namely 
$$\int_X\langle T_1\wedge...\wedge T_n\rangle=\langle\a^n\rangle$$
for every set of positive currents $T_1,...,T_n\in\a$ with full Monge-Amp{\`e}re mass, which is the content of Corollary~\ref{cor:special} above. This fact is non-obvious even when $\a$ is K{\"a}hler, which is the case considered in~\cite{Din2}, and seems to have been implicitely taken for granted there. 

Let thus $T_1=\theta+dd^c \f_1,T_2=\theta+dd^c \f_2 \in\a$ be two positive currents such that
$$\langle T_1^n\rangle=\langle T_2^n\rangle=\mu.$$
As in~\cite{Din1}, we first remark that the log-concavity property of non-pluripolar products (Proposition~\ref{pro:logconcave}) implies that 
$$\langle T_1^k\wedge T_2^{n-k}\rangle\ge\mu$$ 
hence 
\begin{equation}\label{equ:powers}\langle T_1^k\wedge T_2^{n-k}\rangle=\mu
\end{equation}
for $k=0,...,n$ since both measures have same total mass $v:=\vol(\a)$.
\smallskip

\noindent{\bf Step 1}. As in the first part of the proof of Theorem 1 in~\cite{Din2}, we are first going to show by contradiction that there exists $t\in\R$ such that $\f_1=\f_2+t$ $\mu$-almost everywhere. 

Since $\mu$ puts no mass on pluripolar subsets, the set of $t\in\R$ such that $\mu\{\f_1=\f_2+t\}>0$ coincides with the discontinuity locus of the non-decreasing function 
$$t\mapsto\mu\{\f_1<\f_2+t\}$$
and is thus at most countable.  We now assume by contradiction that $\mu\{\f_1=\f_2+t\}<v$ for all $t$, so that we can find $t\in\R$ such that
$$0<\mu\{\f_1<\f_2+t\}<v.$$ 
By monotone convergence, upon slightly pertubing $t$ we can furthermore arrange that 
$$\mu\{\f_1=\f_2+t\}=0.$$
Replacing $\f_2$ by $\f_2+t$, we can finally assume that $t=0$, so that both $\mu\{\f_1<\f_2\}$ and $\mu\{\f_2<\f_1\}$ are less than $v$ while $\mu\{\f_1=\f_2\}=0$. We can thus pick $\e>0$ small enough such that the mass of $(1+\e)^n\mu$ on both $\{\f_1<\f_2\}$ and $\{\f_2<\f_1\}$ is still less than $v$. 

By the existence part of Theorem~\ref{thm:MA} (which is already proved!),  we get a positive current $T=\theta+dd^c\f$ such that 
$$\langle T^n\rangle={\bf 1}_{\{\f_1<\f_2\}}(1+\e)^n\mu+{\bf 1}_{\{\f_1>\f_2\}}t\mu$$
where $t>0$ is taken so that the left-hand side has total mass $v$. Since $\f_{\min}$ has minimal singularities, we can assume that $\f\le\f_{\min}$. 
Proposition~\ref{pro:logconcave} thus implies that
$$\langle T\wedge T_j^{n-1}\rangle\ge(1+\e)\mu$$
on $\{\f_1<\f_2\}$. 

We now consider the subsets 
$$O_\d:=\{(1-\d)\f_1+\d\fm<(1-\d)\f_2+\d\f\}\subset\{\f_1<\f_2\}.$$
Since both $T$ and $T_j$ have full Monge-Amp{\`e}re mass, Corollary~\ref{cor:special} yields

\begin{equation}\label{equ:mass}\int_X\langle T_{\min}\wedge T_j^{n-1}\rangle=\langle\a^n\rangle
\end{equation}
and the generalized comparison principle (Proposition~\ref{pro:comp}) thus implies 
$$\int_{O_\d}\langle\left((1-\d)T_2+\d T\right)\wedge T_j^{n-1}\rangle$$
$$\le\int_{O_\d}\langle\left((1-\d)T_1+\d T_{\min}\right)\wedge T_j^{n-1}\rangle$$
for  $j=1,2$. By (\ref{equ:powers}) it follows that
$$\d(1+\e)\int_{O_\d}\mu$$
$$\le\d\int_{O_\d}\langle T\wedge T_j^{n-1}\rangle\le\d\int_{\{\f_1<\f_2\}} \langle T_{\min}\wedge T_j^{n-1}\rangle.$$
Dividing by $\d>0$ and letting $\d$ tend to $0$, we infer
$$(1+\e)\int_{\{\f_1<\f_2\}}\mu\le\int_{\{\f_1<\f_2\}}\langle T_{\min}\wedge T_j^{n-1}\rangle$$
for $j=1,2$ by monotone convergence. 

We can now play the same game with $\{\f_2<\f_1\}$ (and a possibly different $T$) to conclude
$$(1+\e)\int_{\{\f_2<\f_1\}}\mu\le\int_{\{\f_2<\f_1\}}\langle T_{\min}\wedge T_j^{n-1}\rangle,$$
and we finally reach a contradiction by summing these two inequalities, which yields
$$(1+\e)v\le\int_X \langle T_{\min}\wedge T_j^{n-1}\rangle=\langle\a^n\rangle=v$$
by (\ref{equ:mass}), using $\mu\{\f_1=\f_2\}=0$. 
\smallskip

\noindent{\bf Step 2}. At this point we have proved that $\f_1=\f_2$ a.e.~wrt $\MA(\f_1)=\MA(\f_2)$. One might be tempted to use the domination principle to conclude that $\f_2\le\f_1$, hence $\f_1=\f_2$ by symmetry. But Corollary~\ref{cor:dom} cannot be applied here since $\f_2$ doesn't \emph{a priori} have minimal singularities. The rest of the proof will circumvent this difficulty by using some additional cancellation. 

By the existence part of the proof, we can choose a positive current $T=\theta+dd^c\f$ such that $\langle T^n\rangle$ is a smooth positive volume form. We normalize $\f$ so that $\f\le\f_{\min}$. 

Let $O:=\{\f_1<\f_2\}$. We are going to show by descending induction on $m$ that
\begin{equation}\label{equ:zero}
\int_O\langle T_1^k\wedge T_2^l\wedge T^m\rangle=0
\end{equation}
for every $k,l,m$ such that $k+l+m=n$. The result holds for $m=0$ by the first part of the proof, and the result for $m=n$ will imply $\f_2\le\f_1$ a.e.~wrt Lebesgue measure and will conclude the proof by symmetry. 

Let thus $m$ be given, and assume that (\ref{equ:zero}) holds true for all $k,l$ such that $k+l+m=n$.  Pick $k,l$ such that $k+l+m+1=n$ and consider as usual the current with minimal singularities
$$T_1^{(j)}:=\theta_1+dd^c\f_1^{(j)}$$
with 
$$\f_1^{(j)}=\max(\f_1,\f_{\min}-j).$$
We introduce the plurifine open subsets
$$O_{j,\e}:=\{(1-\frac{\e}{j})\f_1+\frac{\e}{j}\f_1^{(j)}+\e<(1-\frac{\e}{j})\f_2+\frac{\e}{j}\f\}.$$
Note that $O_{j,\e}\subset O$ since 
 $$\frac{\e}{j}\f_1^{(j)}+\e\ge\frac{\e}{j}\f_{\min}\ge\frac{\e}{j}\f,$$
 so that both $\langle T_k\wedge T_2^{l+1}\wedge T^m\rangle$ and $\langle T_1^{k+1}\wedge T_2^l\wedge T^m\rangle$ put no mass on $O_{j,\e}$ by induction.
 
Since $T_1,T_2, T$ and $T_1^{(j)}$ all have full Monge-Amp{\`e}re mass, Corollary~\ref{cor:special} yields 
$$\int_X\langle T_1^{(j)}\wedge T_1^k\wedge T_2^l\wedge T^m\rangle=\langle\a^n\rangle.$$
The generalized comparison principle therefore implies
$$\int_{O_{j,\e}}\langle((1-\frac{\e}{j})T_2+\frac{\e}{j}T)\wedge T_1^k\wedge T_2^l\wedge T^m\rangle$$
$$\le\int_{O_{j,\e}}\langle((1-\frac{\e}{j})T_1+\frac{\e}{j}T_1^{(j)})\wedge T_1^k\wedge T_2^l\wedge T^m\rangle.$$
The leading terms on both sides vanish and we end up with
$$\int_{O_{j,\e}}\langle T_1^k\wedge T_2^l\wedge T^{m+1}\rangle$$
$$\le\int_{O}\langle T_1^{(j)}\wedge T_1^k\wedge T_2^l\wedge T^m\rangle.$$
By Corollary~\ref{cor:special} again
$$\int_X\langle T_1^{k+1}\wedge T_2^l\wedge T^m\rangle=\langle\a^n\rangle,$$
thus 
$$\lim_{j\to\infty}\int_O\langle T_1^{(j)}\wedge T_1^k\wedge T_2^l\wedge T^m\rangle=\int_O\langle T_1^{k+1}\wedge T_2^l\wedge T^m\rangle=0.$$
On the other hand we clearly have
$${\bf 1}_{\{\f_1+\e<\f_2\}}\le\liminf_{j\to\infty}{\bf 1}_{O_{j,\e}}$$
pointwise outside a pluripolar set, and Fatou's lemma thus yields
$$\int_{\{\f_1+\e<\f_2\}}\langle T_1^k\wedge T_2^l\wedge T^{m+1}\rangle=0,$$
hence also 
$$\int_O\langle T_1^k\wedge T_2^l\wedge T^{m+1}\rangle=0$$
by monotone convergence. The induction is thus complete, and the proof of uniqueness is finally over.

\section{$L^{\infty}$-a priori estimate}

We now consider the case of measures
with density in $L^{1+\e}$, $\e>0$ with respect to Lebesgue measure.
The goal of this section is to show how to adapt Ko\l{}odziej's pluripotential theoretic approach to the $L^{\infty}$ \emph{a priori} estimates~\cite{Ko1} to the present context of big cohomology classes. 
Fix an arbitrary (smooth positive) volume form $dV$ on $X$. 

Given a smooth representative $\theta$ of the big cohomology class $\a\in H^{1,1}(X,\R)$, we introduce the extremal function
\begin{equation}\label{equ:V}
V_\theta:=\sup\left\{ \psi\,\,\theta\text{-psh}, \psi\le 0\right\}.
\end{equation}
Note that $V_\theta$ has minimal singularities. Of course we have $V_\theta=0$ if (and only if) $\theta$ is semi-positive.

\begin{thm} \label{thm:apriori}
Let $\mu = f d V$ be a positive measure with density $f \in L^{1+\e}$, $\e>0$, such that $\mu(X)= \vol(\a)$. Then the unique 
closed positive current $T \in \a$ such that $\langle T^n\rangle= \mu$ has minimal singularities.

More precisely, there exists a constant $M$ only depending on $\theta$, $dV$ and $\e$ such that the unique $\theta$-psh function $\f$ such that
$$\MA(\f)=\mu$$
and normalized by
$$\sup_X\f=0$$
satisfies 
\begin{equation}\label{equ:apriori} 
\f\ge V_\theta-M \Vert f\Vert_{L^{1+\e}}^{1 
 \slash n}.
\end{equation}
\end{thm} 

The whole section is devoted to the proof of this result.
As mentioned above, we follow Ko\l{}odziej's approach~\cite{Ko1}. After introducing the appropriate Monge-Amp{\`e}re (pre)capacity $\ca$, we will estimate 
the capacity decay of sublevel sets $\{\f<V_\theta-t\}$ as $t\to+\infty$ in order to show
that actually $\ca\{\f<V_\theta-M\}=0$ for some $M>0$ under control as above, which will prove the theorem.

\subsection{The Monge-Amp{\`e}re capacity}

We define the \emph{Monge-Amp{\`e}re capacity} $\ca$ as the upper envelope of the family of all Monge-Amp{\`e}re measures $\MA(\psi)$ with $\psi$ $\theta$-psh such that $V_\theta\le\psi\le V_\theta+1$, that is for each Borel subset $B$ of $X$ we set
\begin{equation}\label{equ:capa}\ca(B):=\sup\left\{\int_B\MA(\psi), V_\theta\le\psi\le V_\theta+1\right\}.
\end{equation}
By definition we have $\ca(B)\le\ca(X)=\vol(\a)$. 

\smallskip

If $K$ is a compact subset of $X$, we define its \emph{extremal function} by  
$$V_{K,\theta}:=\sup\left\{\psi\,\,\theta\text{-psh}, \psi\le 0
\text{ on } K \right \}.$$
 
Note that 
\begin{equation}\label{equ:triv} 
V_\theta=V_{X,\theta}\le V_{K,\theta}
\end{equation} 
by definition. Standard arguments show that the usc regularization $V_{K,\theta}^*$ of $V_{K,\theta}$ is either $\theta$-psh with minimal singularities when $K$ is non-pluripolar, or $V_{K,\theta}^*\equiv+\infty$ when $K$ is pluripolar. It is also a standard matter to show that $\MA(V_{K,\theta}^*)$ is concentrated on $K$ using local solutions to the homogeneous Monge-Amp{\`e}re equation on small balls in the complement of $K$ (cf.~\cite{GZ1}, Theorem 4.2). It follows from the definition of $V_{K,\theta}$ that the following maximum principle holds:
\begin{equation}\label{equ:maximum}
\sup_K\f=\sup_X(\f-V_{K,\theta})
\end{equation}
for any $\theta$-psh function $\f$.

\smallskip

Finally we introduce the \emph{Alexander-Taylor capacity} of $K$ in the following guise: 
\begin{equation}\label{equ:alex}
M_\theta(K):=\sup_XV_{K,\theta}^*\in [0,+\infty],
\end{equation}
so that $M_\theta(K)=+\infty$ iff $K$ is pluripolar. We then have the following comparison theorem. 

\begin{lem}\label{lem:compcap} For every non-pluripolar compact subset $K$ of $X$, we have
$$1\le\left(\frac{\vol(\a)}{\ca(K)}\right)^{1/n}\le\max(1,M_\theta(K)).$$
\end{lem} 
\begin{proof} The left-hand inequality is trivial. In order to prove the right-hand inequality we consider two cases. If $M_\theta(K)\le 1$, then $V_{K,\theta}^*$ is a candidate in the definition (\ref{equ:capa}) of $\ca(K)$ by~(\ref{equ:maximum}). Since $\MA(V_{K,\theta}^*)$ is supported on $K$ we thus have 
$$\ca(K)\ge\int_K\MA(V_{K,\theta}^*)=\int_X\MA(V_{K,\theta}^*)=\vol(\a)$$
and the desired inequality holds in that case.

On the other hand if $M:=M_\theta(K)\ge 1$ we have 
$$V_\theta\le M^{-1}V_{K,\theta}^*+(1-M^{-1})V_\theta\le V_\theta+1$$
by (\ref{equ:triv}) and it follows by definition of the capacity again that 
$$\ca(K)\ge\int_K\MA(M^{-1}V_{K,\theta}^*+(1-M^{-1})V_\theta).$$
But since
$$\MA(M^{-1}V_{K,\theta}^*+(1-M^{-1})V_\theta)\ge M^{-n}\MA(V_{K,\theta}^*)$$ 
we deduce that
$$\int_K\MA(M^{-1}V_{K,\theta}^*+(1-M^{-1})V_\theta)\ge M^{-n}\int_X\MA(V_{K,\theta}^*)=M^{-n}\vol(\a)$$
and the result follows. 
\end{proof}

Using this fact, we will now prove as in~\cite{Ko1} that any measure with $L^{1+\e}$-density is nicely dominated by the capacity. 

\begin{pro}\label{pro:capLp} 
Let $\mu=fdV$ be a positive measure with $L^{1+\e}$ density with respect to 
Lebesgue measure, with $\e>0$.
Then there exists a constant $C$ only depending on $\theta$ and $\mu$ such that 
$$\mu(B)\le C\cdot\ca(B)^2$$
for all Borel subsets $B$ of $X$. In fact we can take
$$C:=\e^{-2n}A\Vert f\Vert_{L^{1+\e}(dV)}$$
for a constant $A$ only depending on $\theta$ and  $dV$. 
\end{pro}

\begin{proof} By inner regularity of $\mu$ it is enough to consider the case where $B=K$ is compact. We can also assume that $K$ is non-pluripolar since $\mu(K)=0$ otherwise and the inequality is then trivial. 

Now set
\begin{equation}\label{equ:lelong}\nu_\theta:=2\sup_{T,x}\nu(T,x)
\end{equation}
the supremum ranging over all positive currents $T\in\a$ and all $x\in X$ and $\nu(T,x)$ denoting the Lelong number of $T$ at $x$. Since all Lelong numbers of $\nu_\theta^{-1}T$ are $\le 1/2<1$ for each positive current $T\in\a$, the uniform version of Skoda's integrability theorem together with the compactness of normalized $\theta$-psh functions yields a constant $C_\theta>0$ only depending on $dV$ and $\theta$ such that
$$\int_X\exp(-\nu_\theta^{-1}\psi)dV\le C_\theta$$
for all $\theta$-psh functions $\psi$ normalized by $\sup_X\psi=0$ (see~\cite{Ze}). Applying this to $\psi=V_{K,\theta}^*-M_\theta(K)$ (which has the right normalization by (\ref{equ:alex})) we get
$$\int_X\exp(-\nu_\theta^{-1}V_{K,\theta}^*)dV\le C_\theta\exp(-\nu_{\theta}^{-1} M_\theta(K)).$$
On the other hand we have $V_{K,\theta}^*\le 0$ on $K$ a.e.~with respect to Lebesgue measure, hence
\begin{equation}\label{equ:vol}\vol(K)\le C_\theta\exp(-\nu_\theta^{-1}M_\theta(K)).
\end{equation}
H{\"o}lder's inequality yields  
\begin{equation}\label{equ:holder}
\mu(K)\le\Vert f\Vert_{L^{1+\e}(dV)}\vol(K)^{\e/1+\e}.
\end{equation}
We may also assume that $M_\theta(K)\ge 1$. Otherwise Lemma~\ref{lem:compcap} implies $\ca(K)=\vol(\a)$, and the result is thus clear in that case. By Lemma~\ref{lem:compcap}, (\ref{equ:vol}) and (\ref{equ:holder}) together we thus get
$$\mu(K)\le\Vert f\Vert_{L^{1+\e}(dV)}\exp\left(-\frac{\e}{(1+\e)\nu_\theta}\left(\frac{\ca(K)}{\vol(\a)}\right)^{-1/n}\right)$$ 
times $C_\theta^{\e/1+\e}$ and the result follows since
$\exp(-t^{-1/n})=O(t^2)$ when $t\to 0_+$. 
\end{proof}

\subsection{Proof of Theorem~\ref{thm:apriori}}
We first apply the comparison principle to get:
\begin{lem} \label{lem:majoCap}
Let $\f$ be a $\theta$-psh function with full Monge-Amp{\`e}re mass. Then for all $t>0$ and $0<\d<1$ we have
$$\ca\{\f<V_\theta-t-\d\} \leq \d^{-n}\int_{\{\f<V_\theta-t\}}\MA(\f).
$$
\end{lem}
\begin{proof}
Let $\psi$ be a $\theta$-psh function such that $V_\theta\le\p\le V_\theta+1$. We then have
$$\{\f<V_\theta-t-\d\}\subset\{\f<\d\p+(1-\d)V_\theta-t-\d\}\subset\{\f<V_\theta-t\}.$$
Since $\d^n\MA(\p)\le\MA(\d\p+(1-\d)V_\theta)$ and $\f$ has finite Monge-Amp{\`e}re energy, Proposition~\ref{pro:comp} yields
$$\d^n\int_{\{\f<V_\theta-t-\d\}}\MA(\p)\le\int_{\{\f<\d\p+(1-\d)V_\theta-t-\d\}}\MA(\d\p+(1-\d)V_\theta)$$
$$\le\int_{\{\f<\d\p+(1-\d)V_\theta-t-\d\}}\MA(\f)\le\int_{\{\f<V_\theta-t\}}\MA(\f)$$
and the proof is complete.
\end{proof}

Now set
$$g(t):=\left(\ca\{\f<V_\theta-t\}\right)^{1/n}.$$ 
Our goal is to show that $g(M)=0$ for some $M$ under control. Indeed we will then have $\f \ge V_\theta-M$ on $X \setminus P$ for some Borel subset $P$ such that 
where $\ca(P)=0$. But it then follows in particular from
Proposition~\ref{pro:capLp} (applied to the Lebesgue measure itself) that $P$ has Lebesgue measure zero hence $\f \ge V_\theta-M$ will hold everywhere.

Now since $\MA(\f)=\mu$ it follows from Proposition~\ref{pro:capLp} and Lemma~\ref{lem:majoCap} that
$$
g(t+\d) \leq \frac{C^{1/n}}{\d} g(t)^2
\; \; \text{ for all } t>0 \text{ and } 0<\d<1.
$$
 We can thus apply Lemma 2.3 in \cite{EGZ1} which yields 
$g(M)=0$ for $M:=t_0+4C^{1/n}$ (see also ~\cite{BGZ}).
Here $t_0>0$ has to be chosen so that 
$$g(t_0)<\frac{1}{2C^{1/n}}.$$ 
Lemma \ref{lem:majoCap} (with $\d=1$) implies that
$$
g(t)^n\le\mu\{\f<V_\theta-t+1\}$$
$$\le\frac{1}{t-1}\int_X|V_\theta-\f|fdV
\le\frac{1}{t-1} \Vert f\Vert_{L^{1+\e}(dV)}(\Vert\f\Vert_{L^{1+1/\e}(dV)}+\Vert V_\theta\Vert_{L^{1+1/\e}(dV)})
$$
by H{\"o}lder's inequality. Since $\f$ and $V_\theta$ both belong to the compact set of $\theta$-psh functions normalized by $\sup_X \f=0$, their $L^{1+1/\e}(dV)$-norms are bounded by a constant $C_2$ only depending on $\theta$, $dV$ and $\e$. It is thus enough to take
$$t_0>1+2^{n-1}C_2C\Vert f\Vert_{L^{1+\e}(dV)}.$$

\begin{rem}\label{rem:uniform} 
{\it We have already mentioned that $\vol(\a)$ is continuous on the big cone. On the other hand it is easy to see that $\nu_\theta$ remains bounded as long as $\theta$ is $C^2$-bounded. As a consequence if $\mu=f dV$ is fixed we see that the constant $C$ in Proposition~\ref{pro:capLp} can be taken uniform in a small $C^2$-neighborhood of a given form $\theta_0$. We similarly conclude that $M$ in Theorem~\ref{thm:apriori} is uniform around a fixed form $\theta_0$, $\mu$ being fixed.  }
\end{rem}

\begin{rem} 
{\it As in~\cite{Ko1} Theorem~\ref{thm:apriori} holds more generally when the density $f$ belongs to some Orlicz class. Indeed it suffices to replace (\ref{equ:holder}) by the appropriate Orlicz-type inequality in the proof of Proposition~\ref{pro:capLp}. }
\end{rem}
\smallskip

\section{Regularity of solutions: the case of nef classes}
Let $\a$ be a big cohomology class and let $\mu$ be a smooth positive volume form of total mass equal to $\vol(\a)$. Let $T\in\a$ be the unique positive current such that $\langle T^n\rangle=\mu$. By Theorem~\ref{thm:apriori} $T$ has minimal singularities, which implies in particular that $T$ has locally bounded potentials on the Zariski open subset $\Amp(\a)$, the ample locus of $\a$. Note that the equation $\langle T^n\rangle=\mu$ then simply means that $T^n$, which is well-defined on $\Amp(\a)$ by Bedford-Taylor, satisfies $T^n=\mu$ on this open subset. 

If $\a$ is a K{\"a}hler class, then $\Amp(\a)=X$ and Yau's
theorem~\cite{Y} implies that $T$ is smooth. For an arbitrary big
class $\a$ the expectation is that $T$ is smooth on $\Amp(\a)$. By
Evans-Trudinger's general regularity theory for fully non-linear
elliptic equations~\cite{Eva,Tru} it is enough to show that $T$ has
$L^\infty_{loc}$ coefficients on $\Amp(\a)$, or equivalently that the
trace measure of $T$ has $L^\infty_{loc}$-density with respect to
Lebesgue measure (cf.the discussion in~\cite{Blo1}). We are
unfortunately unable to prove this is general, but we can handle the
following special case. The arguments are similar to those used by
Sugiyama in~\cite{Sug}. 

\begin{thm}\label{thm:regnef} Let $\a$ be a nef and big class. Let $\mu$ be a smooth
  positive volume of total mass equal to $\vol(\a)=\a^n$. Then the
  positive current $T\in\a$ such that $\langle T^n\rangle=\mu$ is smooth on $\Amp(\a)$.
\end{thm}

\begin{proof} Write $T=\theta+dd^c\f$ with $\sup_X\f=0$. The $\theta$-psh
  function $\f$ is locally bounded on $\Omega:=\Amp(\a)$ and satisfies 
$$(\theta+dd^c\f)^n=\mu$$
there. If $\pi:X'\to X$ is a modification that is isomorphic over $\Omega$, then $\f':=\pi^*\f$,
$\theta'=\pi ^*\theta$ and $\mu':=\pi^*\mu$ obviously satisfy
$$(\theta'+dd^c\f')^n=\mu'$$
on $\Omega':=\pi^{-1}\Omega$. Note also that $\f'$ is a $\theta'$-psh
function with minimal singularities by
Proposition~\ref{pro:tire_min}. Now $\Amp(\a)$ is covered by Zariski open subsets of the form
$X-\pi(E)$ where $\pi$ is as above, $E$ is an effective
$\R$-divisor on $X'$ and $\theta'$ is cohomologous to $\omega+[E]$ for
some K{\"a}hler form $\omega$ on $X'$. 
We thus fix such a data, and our goal is to show that $\Delta\f'$
belongs to $L^\infty_{loc}$ on $X'-E$ (the Laplacian being computed with respect to $\omega$). We can find a quasi-psh function $\f_E$ on $X'$ such that 
$$[E]=\theta'-\omega+dd^c\f_E,$$
so that $\f_E$ is smooth on $X'-E$ and satisfies $\theta'=\omega-dd^c\f_E$ there. Since $\f_E$ is in particular $\theta$-psh and
$\f'$ has minimal singularities, we have $\f_E\le\f'+O(1)$. Upon
replacing $E$ by $(1+\e)E$ with $0<\e\ll 1$ in the above construction and shifting $\f_E$ by a constant we can in fact assume that 
\begin{equation}\label{equ:poles}
\f'\ge(1-\e)\f_E
\end{equation}
for some $\e>0$. 

On the other hand let's write
$$\mu'=e^F\omega^n.$$
The function $e^F$ is smooth but vanishes on the critical locus of
$\pi$. Indeed $F$ is quasi-psh (in particular bounded from above) and
$dd^cF$ is equal to the integration current on the relative canonical
divisor of $\pi$ modulo a smooth form. In particular the Laplacian
$\Delta F$ (again with respect to $\omega$) is \emph{globally}
bounded on $X'-E$. 

For each $t\ge 0$ we consider the smooth $(1,1)$-form 
$$\theta_t:=\theta'+t\omega.$$
Note that a $\theta_t$-psh function \emph{a fortiori} is $\theta_s$-psh for $s\ge t$ since $\omega\ge 0$. 
We set $V_t:=V_{\theta_t}$ where the latter is the extremal
$\theta_t$-psh function defined by (\ref{equ:V}). The proof of the following lemma is straightforward.

\begin{lem}\label{lem:decrease} The quasi-psh functions $V_t$ decrease pointwise to $V_0$ as $t\to 0$.
\end{lem}
Since $V_0$ is a $\theta_0$-psh function with
minimal singularities (\ref{equ:poles}) together with Lemma~\ref{lem:decrease} imply that $V_t-\f_E$ tends to $+\infty$ near $E$ for each $t\ge 0$. 

By Theorem~\ref{thm:apriori} for each $t\ge 0$ there exists a unique $\theta_t$-psh function $\f_t$ with minimal singularities such that 
\begin{equation}\label{equ:smooth}
(\theta_t+dd^c\f_t)^n=e^F\omega^n
\end{equation}
normalized by $\sup_X\f_t=0$. By Theorem~\ref{thm:apriori} and Remark~\ref{rem:uniform}, there exists $M>0$ such that 
\begin{equation}\label{equ:minoree}\f_t\ge V_t-M 
\end{equation}
for all $t\ge 0$. We thus have $\f_0=\f'$. 
The class of $\theta_t$ is K{\"a}hler for $t>0$ since the class of
$\theta$ is nef. Since $F$ has analytic singularities, Theorem 3 on
p.374 of~\cite{Y} implies that $\f_t$ is smooth on $X'-E$ and
$\Delta\f_t$ is globally bounded on $X'$ for each
$t>0$ (but of course no uniformity is claimed with respect to $t$). 

\begin{lem}\label{lem:limit} The normalized solutions $\f_t$ satisfy 
$$\lim_{t\to 0}\f_t=\f_0.$$
\end{lem}
\begin{proof} The normalized quasi-psh functions $\f_t$ live in a compact subset of $L^1(X)$. By the uniqueness part of Theorem~\ref{thm:MA} it is thus enough to show that any limit $\p$ of a sequence $\f_k:=\f_{t_k}$ with $t_k\to 0$ satisfies $\langle\theta_0+dd^c\p)^n\rangle=e^F\omega^n$ and $\sup_X\p=0$. The latter property follows from Hartog's lemma. To prove the former, we introduce 
$$\p_k:=(\sup_{j\ge k}\f_j)^*$$
so that $\p_k$ is $\theta_k$-psh and decreases pointwise to $\p$. As in Proposition~\ref{pro:monge_max} we get
$$(\theta_k+dd^c\p_k)^n\ge e^F\omega^n$$
on $X'-E$ and the result follows by continuity of the Monge-Amp{\`e}re operator along decreasing sequences of bounded psh functions.
\end{proof}

We are now going to prove that $\Delta\f_t$ is uniformly bounded on
compact subsets of $X'-E$, which will imply that $\Delta\f_0$ is $L^\infty_{loc}$ on $X'-E$ by Lemma~\ref{lem:limit}. This will be accomplished by using Yau's
poinwise computations~\cite{Y}, p.350. In order to do so, we rely on
Tsuji's trick~\cite{Tsu88}: we introduce 
$$u_t:=\f_t-\f_E$$
which is smooth on $X'-E$ and satisfies 
$$(\omega+dd^cu_t)^n=\mu$$
there, since $\theta'=\omega-dd^c\f_E$ on $X'-E$. Note that $u_t$ is \emph{not} quasi-psh on $X'$. Indeed we have $u_t\to+\infty$ near $E$ since $\f_t=V_t+O(1)$ on $X'$. Since $\Delta\f_E$ is
globally bounded on $X'-E$ as was already noted, to bound $\Delta\f_t$ is equivalent to bounding $\Delta u_t$. 

We now basically follow the argument on p.350
of~\cite{Y}. By inequalities (2.18) and (2.20) on p.350 of~\cite{Y}
(which only depend on pointwise computations)  we have
$$e^{Au_t}\Delta_t\left(e^{-Au_t}(n+\Delta u_t)\right)$$
$$\ge-An(n+\Delta u_t)+(A+b_\omega)e^{-\frac{F}{n-1}}(n+\Delta\psi)^{\frac{n}{n-1}}+\Delta F-n^2 b_\omega$$
for every $A>0$ such that $A+b_\omega>1$. Here $\Delta$ and
$\Delta_t$ respectively denote the (negative) Laplacians associated
with the K{\"a}hler forms $\omega$ and $\omega+dd^c u_t$ and $b_\omega$
denotes the (pointwise) minimal holomorphic bisectional curvature of
$\omega$. Now $b_\omega$ is globally bounded, $F$ is bounded from
above and $\Delta F$ is globally bounded on $X'-E$. We can thus find $B,C,D>0$ (independent of $t$) such that 
\begin{equation}\label{equ:yaunef} e^{Au_t}\Delta_t\left(e^{-Au_t}(n+\Delta u_t)\right)
\ge-B(n+\Delta u_t)+C(n+\Delta u_t)^{\frac{n}{n-1}}-D\end{equation}
holds on $X'-E$. Since $\Delta u_t=\Delta\f_t-\Delta\f_E$ is smooth
and globally bounded on $X'-E$ and since $u_t$ tends to $+\infty$ near $E$, we see that
$$e^{-Cu_t}(n+\Delta u_t)$$
achieves its maximum on $X'-E$ at some point $x_t\in X'-E$. Inequality
(\ref{equ:yaunef}) applied at $x_t$ thus yields
$$C(n+\Delta u_t)^{\frac{n}{n-1}}\le B(n+\Delta u_t)+D$$
at $x_t$. Since $n/n-1>1$, we thus see that there exists $C_1$ (independent of $t$) such that
$$(n+\Delta u_t)(x_t)\le C_1.$$
We thus get
\begin{equation}\label{equ:borneinf} n+\Delta u_t\le C_1\exp\left(C\left(u_t-\inf_{X'-E}u_t\right)\right)\end{equation}
by definition of $x_t$ (compare (2.24) on p.351 of~\cite{Y}). By (\ref{equ:poles}) and (\ref{equ:minoree}) we have
$$-M\le V_t-\f_E-M\le\f_t-\f_E=u_t\le(\sup_XV_1)-\f_E$$
for $0\le t\le 1$ which shows that $u_t-\inf_{X'-E}u_t$ is uniformly bounded from above on compact subsets of $X'-E$, and the proof is thus complete in view of (\ref{equ:borneinf}) 
\end{proof}

It is a consequence of Demailly's regularization theorem that currents with minimal singularities $T\in\a$ in a nef and big class have identically zero Lelong numbers (cf.~\cite{Bou2}). Here is an example where such currents however have poles. 

\begin{ex} \label{ex:bignef}
Start with a famous example due to Serre: let $E$ be the (flat, but not
unitary flat) rank 2 vector bundle over the elliptic curve $C:=\C/\Z[i]$
associated to the representation $\pi_{1}(C)=\Z[i]\rightarrow SL(2,\C)$
sending $1$ to the identity and $i$ to $\left(\begin{array}{cc}
1 & 1\\
0 & 1\end{array}\right)$. The ruled surface $S:=\PP(E)\rightarrow C$ of hyperplanes of $E$
has a section $C'$ with flat normal bundle, which lies in the linear
system $|\cO_{E}(1)|$. The original point of this construction of Serre was that $X-C'$
is Stein but not affine, and the reason for that is that $C'$ is
rigid in $X$ despite having a non-negative normal bundle. In fact, Demailly-Peternell-Schneider have proved~\cite{DPS} that $C'$ is rigid in the very strong sense that 
the only closed positive current cohomologous to $C'$ is $C'$ itself. 

Now let $X:=\PP(V)$ be the projective bundle of hyperplanes in $V:=E\oplus A$,
where $A$ is a given ample line bundle on $C$, and let $L:=\cO_{V}(1)$
be the tautological bundle. The line bundle $L$ is nef since $E$ and
$A$ are nef, and it is also big since $A$ is big. It is easy to
show that the non-ample locus of $L$ is exactly $S=\PP(E)\subset\PP(V)=X$.
But the restriction of $L$ to $S$ is $\cO_{E}(1)$, and positive currents  can be restricted to any subvariety not entirely contained in their
polar set. It follows that any positive current in the nef and big class $\a:=c_1(L)$ has poles along $C'$. 
\end{ex}

\section{Singular K{\"a}hler-Einstein metrics}

\subsection{More Monge-Amp{\`e}re equations}
In this section we show how to use the apriori estimate of Theorem~\ref{thm:apriori} and Schauder's fixed point theorem to solve Monge-Amp{\`e}re equations with non-constant right-hand side. 

We let as before $dV$ be a smooth volume form on $X$ and $\theta$ be a smooth representative of a big cohomology class. 

\begin{thm}\label{thm:moreMA}
There exists a unique $\theta$-psh function $\f$ with full Monge-Amp{\`e}re mass such that 
$$\MA(\f)=e^\f dV.$$
Furthermore $\f$ has minimal singularities. 
\end{thm}
Since $\f$ is bounded from above, the measure $e^\f\mu$ has $L^\infty$-density with respect to Lebesgue measure, hence Theorem~\ref{thm:apriori} implies that any $\theta$-psh solution $\f$ with full Monge-Amp{\`e}re mass has minimal singularities. 

We now proceed with the existence proof, and uniqueness will be taken care of by Proposition~\ref{pro:increase} below. 

\begin{proof} We can assume that $dV$ has total mass $1$. Let $\cC$ be the compact convex subset of $L^1(X)$ consisting of all $\theta$-psh functions $\p$ normalized by $\sup_X\p=0$. By compactness there exists $C>0$ such that 
$$\int_X\p dV\ge -C$$ 
for all $\p\in\cC$ (cf.~\cite{GZ1}, Proposition 1.7).  On the other hand since $dV$ is a probability measure we get by convexity
\begin{equation}\label{equ:bornee}\log\int_Xe^{\p}dV\ge -C
\end{equation}
for all $\p\in\cC$. 

As noticed above, since $e^{\p}dV$ has $L^\infty$ density with respect to Lebesgue measure, it follows from Theorem \ref{thm:MA} that for each $\p\in\cC$ there exists a unique $\Phi(\p)\in\cC$ such that
$$\MA(\Phi(\p))=e^{\p+c_\p}dV,$$
where 
$$c_{\p}:=\log\frac{\vol(\a)}{\int_Xe^{\p}dV}$$ 
is a normalizing constant making the total masses fit. 
Note that 
$$c_\psi\le\log\vol(\a)+C$$
for all $\p\in\cC$ by (\ref{equ:bornee}). Theorem~\ref{thm:apriori} therefore gives us a constant $M>0$ independent of $\p\in\cC$ such that 
$$\Phi(\p)\ge V_\theta-M.$$

\begin{lem} The mapping $\Phi:\cC\to\cC$ is continuous.
\end{lem}
\begin{proof} Let $\psi_j\to\p$ be a convergent sequence in $\cC$. Upon extracting we can assume by compactness that 
$$\f_j:=\Phi(\p_j)\to\f$$ 
for some $\f\in\cC$, and also that $\psi_j\to\p$ almost everywhere. By definition of $c_\p$, this implies $c_{\p_j}\to c_\p$ by dominated convergence. We claim that $\f$ has full Monge-Amp{\`e}re mass and satisfies $\MA(\f)=e^{\p+c_\p}dV$. By uniqueness we will then get $\f=\Phi(\p)$ as desired. 

Indeed since $\Phi(\p_j)\ge V_\theta-M$ for all $j$, it follows that also $\f\ge V_\theta-M$, and in particular $\f$ has full Monge-Amp{\`e}re mass. Corollary~\ref{cor:monge_max} thus yields $\MA(\f)\ge e^{\p+c_\p}\mu$, and equality follows as desired since the total masses are both equal to $\vol(\a)$.  

\end{proof}

By Schauder's fixed point theorem, $\Phi$ has a fixed point $\p\in\cC$, 
and we then get a solution by setting $\f:=\p+\log c_\p$.
\end{proof}

\begin{pro}\label{pro:increase} If $\f_i$, $i=1,2$ are two $\theta$-psh functions with small unbounded locus such that
$$\MA(\f_i)=e^{\f_i }dV.$$
Then $\f_1\le\f_2+O(1)$ already implies $\f_1\le\f_2$. 
\end{pro} 
\begin{proof}By the refined comparison principle for functions with small unbounded locus explained in Remark~\ref{rem:refined}, we get 
$$\int_{\{\f_2<\f_1-\e\}}\MA(\f_1)\le\int_{\{\f_2<\f_1-\e\}}\MA(\f_2),$$
that is 
$$\int_{\{\f_2<\f_1-\e\}}e^{\f_1}dV\le\int_{\{\f_2<\f_1-\e\}}e^{\f_2}dV.$$
This implies that $\f_2\ge\f_1-\e$ almost everywhere, hence everywhere and the result follows.
\end{proof}
As an immediate corollary to Theorem~\ref{thm:moreMA} we get:

\begin{cor}\label{cor:KE} Let $X$ be a projective manifold of general
  type, that is with a big canonical bundle $K_X$. Then there exists a
  unique singular non-negatively curved metric $e^{-\phi_{KE}}$ on
  $K_X$ satisfying the K{\"a}hler-Einstein equation
$$\langle(dd^c\phi_{KE})^n\rangle= e^{\phi_{KE}}$$
and such that
$$\int_X e^{\phi_{KE}}=\vol(X).$$ 
Furthermore $\phi_{KE}$ has minimal singularities.
\end{cor}

\subsection{Comparison with previous results}
We compare Corollary~\ref{cor:KE} with previous results and discuss
the difficulties raised by the regularity issue on $\Amp(K_X)$.
\smallskip
 
\noindent{\bf Ample case.} When $K_X$ is ample we have $\Amp(K_X)=X$
and Aubin-Yau's theorem~\cite{Aub76,Y} can of
course be reformulated by saying that $\phi_{KE}$ constructed in
Corollary~\ref{cor:KE} is smooth.
\smallskip
  
\noindent{\bf Nef and big case.} Assume that $K_X$ is nef and big. It is then
semiample by the base-point-free theorem, so that $\Amp(K_X)$ coincides
with the degeneracy locus of the birational morphism
$$X\to\xcan:=\mathrm{Proj }R(K_X)$$
and minimal singularities means locally bounded. 
Theorem 1 of~\cite{Tsu88} constructs a "canonical" psh weight $\phi$ on $K_X$
which is smooth outside $\Amp(K_X)$ and satisfies
$\langle(dd^c\phi)^n\rangle=e^{\phi}$. Theorem 3.8 of~\cite{Sug} yields the additional information that $\phi$ has
globally bounded Laplacian on $X$ (hence in $C^{2-\e}(X)$ for each $\e>0$).
\smallskip

\noindent{\bf Big case.} If we are willing to use finite generation of the
canonical ring $R(K_X)$ 
\cite{BCHM} (see also~\cite{Ts99,S} for more analytic approaches) let $\mu:Y\to X$,
$\nu:Y\to\xcan$ be a resolution of the graph of
$X\dashrightarrow\xcan$. By the negativity lemma (\cite{KM} Lemma 3.39) we have $$\mu^*K_X=\nu^*K_{\xcan}+E$$
where $E$ is an effective $\nu$-exceptional $\Q$-divisor. This
decomposition is thus the Zariski decomposition of $K_X$, and it
follows that every positive current $T$ in $c_1(K_X)$ satisfies
$$\mu^*T=\nu^*S+[E]$$
for a unique positive current $S$ in $c_1(K_{\xcan})$ (cf.~for
instance~\cite{Bou2}). It is then immediate to see that the
K{\"a}hler-Einstein metric $e^{-\phi_{KE}}$ on $K_X$ constructed in
Corollary~\ref{cor:KE} corresponds in this way to the the K{\"a}hler-Einstein metric on $K_{\xcan}$ constructed in~\cite{EGZ1}, Theorem 7.8. 

\smallskip

Suppose now that we don't use the existence of the canonical model
$\xcan$. Theorem 5.1
of~\cite{Ts06} claims the existence of a psh weight $\phi$ on $K_X$ such that $\phi$ is smooth on a
Zariski open subset $U$ of $X$, $(dd^c\phi)^n=e^\phi$ holds on $U$ and $\phi$ is an \emph{analytic Zariski decomposition}
(AZD). This property means that the length with respect to $\phi$ of
every pluricanonical section is in $L^2$ on $X$ (hence is implied by
the stronger condition that such length functions are bounded).

Tsuji's argument essentially proceeds as follows. As in the proof of
Theorem~\ref{thm:MA} above we consider approximate Zariski decompositions,
that is we let $\pi_k:X_k\to X$ be an increasing sequence
of modifications such that 
$$\pi_k^*K=A_k+E_k$$ 
with $A_k$ ample, $E_k$
effective and $E_{k+1}<E_k$ (on $X_{k+1}$), and such that
$$\lim_{k\to\infty}\vol(A_k)=\vol(K_X).$$
We may assume that $\pi_0$ is
the identity. 

For each $k$, we can
use~\cite{Y} as in the beginning of the proof of Theorem~\ref{thm:MA}
to get a psh weight $\phi_k$ on $K_X$ such that
$$\pi_k^*dd^c\phi_k=\omega_k+[E_k]$$ 
for some K{\"a}hler form $\omega_k$ in the class of $A_k$ and
$(dd^c\phi_k)^n=e^{\phi_k}$ holds on $X-E_1$. 

Tsuji shows that $\phi_k$ is non-decreasing and locally uniformly
bounded from above by using the comparison
principle (compare Proposition~\ref{pro:increase}). The weight 
$\phi:=(\sup_k\phi_k)^*$ is thus psh and satisfies  
$$\phi_0\le\phi_k\le\phi$$
for all $k$. It follows by construction that the length with respect
to $\phi$ of every pluricanonical section is $L^\infty$ on $X$, so
that $\phi$ is in particular an AZD. These arguments of Tsuji have also
been expanded in Section 4.3 of~\cite{ST}. 

Note that by monotone convergence we have
$$
\int_Xe^\phi=\lim_{k\to\infty}\int_Xe^{\phi_k}.
$$
Now $\langle(dd^c\phi_k)^n\rangle=e^{\phi_k}$ and it is easy to see
that
$$\int_X\langle(dd^c\phi_k)^n\rangle=\vol(A_k).$$
It thus follows that $\int_Xe^\phi=\vol(K_X)$, so that the psh weight
$\phi$ constructed in~\cite{Ts06,ST} coincides with $\phi_{KE}$ of
Corollary~\ref{cor:KE}. As was already noted, it follows by
construction that the length wrt $\phi$ of any section $\sigma\in
H^0(kL)$ is bounded, which implies in particular that $\phi$ is an AZD
in the sense of Tsuji. On the other hand Corollary~\ref{cor:KE} says
that $\phi_{KE}$ has \emph{minimal singularities}. As we explain below, the
latter condition is strictly stronger than the former for a
general big line bundle $L$ (but \emph{a posteriori} not for $K_X$ thanks to finite
generation of the canonical ring). 

\smallskip

We now discuss Tsuji's approach to the smoothness of $\phi$ on
$X-E_0$, which unfortunately seems to encounter a difficulty that we were not able to fix. 
As in the proof of Theorem~\ref{thm:regnef} above, set
$u_k:=\phi_k-\phi_0$, which is smooth on $X-E_0$ and satisfies
$$(\omega_0+dd^c u_k)^n=e^{u_k+F}\omega_0^n$$
on $X-E_0$ with $F$ such that $e^F\omega_0^n=e^{\phi_0}$. Following
the arguments on p.350 of~\cite{Y} as we did in the proof of
Theorem~\ref{thm:regnef}, Tsuji claims that $\Delta u_k$ is
uniformly bounded on compact subsets of $X-E_0$ (which would indeed
imply that $\phi$ is smooth on $X-E_0$). 

In order to apply the maximum principle to (\ref{equ:yaunef}) one needs
to choose $A>0$ such that 
$$e^{-Au_k}(n+\Delta u_k)$$ 
is bounded on
$X-E_0$. But as opposed to the nef case the positive form
$\omega_0+dd^cu_k$ is just the push-forward of a form with bounded
coefficients, so that it \emph{a priori} acquires poles along $X-E_0$ 
(as was pointed out to us by  Demailly) and it
doesn't seem to be obvious to show that the poles in question can be controled
by $e^{Au_k}$ for some uniform $A>0$. 

\smallskip

\subsection{Analytic Zariski decompositions vs.~minimal singularities}

Let $L$ be a big line bundle on $X$. For each $m$ we choose basis of
sections $\sigma_j^{(m)}$ of $H^0(m!L)$ and let
$$\phi_m:=\frac{1}{m!}\log\left(\sum_j|\sigma_j^{(m)}|\right)$$
be the associated psh weight on $L$. Upon scaling the sections we can arrange that 
\begin{equation}\label{equ:decrease}\phi_m\le\phi_{m+1}\le\psi
\end{equation}
for some weight $\psi$ with
minimal singularities on $L$. 

On the other hand let $\sum_{m\ge 1}\e_m$ be a convergent series of
positive numbers. Then one
can consider the psh weight
$$\rho:=\log\left(\sum_{m\ge 1}\e_m e^{\phi_m}\right).$$ 
This kind of weight was first introduced by Siu in~\cite{Siu2}. Note
that we have
$$\rho\ge\phi_m+\log\e_m$$
for each $m$, which implies that the length wrt $\rho$ of any section $\sigma\in
H^0(kL)$ is bounded on $X$. 

The following observation was explained to us by J.-P.~Demailly.

\begin{pro}\label{pro:azd} The psh weight $\rho$ of Siu-type has minimal singularities
  iff the algebra 
$$R(L)=\oplus_{k\ge 0}H^0(kL)$$
is finitely generated.
\end{pro}
In order to appreciate this fact, it should be recalled $R(L)$ is generally \emph{not} finitely generated when $L$ is big. In fact if $L$ is nef and big for instance, then $R(L)$ is finitely generated iff some multiple of $L$ is base-point free. Classical constructions of Zariski and Cutkosky yield examples in dimension 2 (cf.~\cite{Laz}). Note that Example~\ref{ex:bignef} also yields in particular an example (in dimension 3). 
\begin{proof} One direction is clear. If we conversely assume that $R(L)$ is not finitely
  generated, then Lemma~\ref{lem:fg} below implies that no $\phi_m$
  can have minimal
  singularities. Therefore given any sequence $C_m>0$ we can find
 for each $m$ some point $x_m\in X$ such that
\begin{equation}\label{equ:negative}\phi_m(x_m)\le\psi(x_m)-C_m.\end{equation}
It thus follows from (\ref{equ:decrease}) and (\ref{equ:negative}) that
$$e^{\rho(x_m)}\le(\e_1+...+\e_m)e^{-C_m}e^{\psi(x_m)}+\d_me^{\psi(x_m)}$$
with 
$$\d_m:=\sum_{l>m}\e_l.$$
If we choose $C_m\to+\infty$ fast enough to ensure that
$$(\e_1+...+\e_m)e^{-C_m}\le\d_m$$
then we get
$$(\rho-\psi)(x_m)\le\log(2\d_m)$$
for all $m$, which shows that we cannot have $\rho=\psi+O(1)$. 
\end{proof}
The following fact is completely standard modulo the language used.  
\begin{lem}\label{lem:fg} $R(L)$ is finitely generated iff there exists $m$ such that for each $l\ge m$ we have
  $\phi_l=\phi_m+O(1)$.
\end{lem}
\begin{proof} Let $\pi:X_m\to X$ be an increasing sequence of log-resolutions of the base
  scheme of $|m!L|$, so that $\pi^*m!L=M_m+F_m$. The assumption
  amounts to saying that $F_l$ coincides with the pull-back of $F_m$
  to $X_l$, which is in turn equivalent to saying that every section
  of $H^0(X_m,l!\pi^*L)$ vanishes along $F_m$. But this implies that the canonical inclusion
$$H^0(kM_m)\to H^0(k\pi^*m!L)$$
is an isomorphism for each $k$, and the result follows since $R(L)$ is
finitely generated iff $R(aL)$ is finitely generated for some integer
$a$.
\end{proof}

\end{document}